\documentclass{amsart}

\usepackage{kotex}
\usepackage{amsmath}
\usepackage{amsfonts}
\usepackage{amscd}

\usepackage{amssymb}
\usepackage{graphicx}
\usepackage{caption}
\usepackage{subcaption}
\usepackage{dsfont}
\usepackage{color}
\usepackage{hyperref}
\usepackage{bbm}
\usepackage{a4wide}
\usepackage{sseq}
\usepackage{tikz-cd}
\usepackage[abs]{overpic}

\usepackage{mathabx}

\newtheorem{theorem}{Theorem}[section]
\newtheorem{lemma}[theorem]{Lemma}

\newtheorem{proposition}[theorem]{Proposition}
\newtheorem{corollary}[theorem]{Corollary}

\newtheorem*{question*}{Question}

\theoremstyle{definition}
\newtheorem*{definition*}{Definition}
\newtheorem{definition}[theorem]{Definition}

\newtheorem*{example*}{Example}
\newtheorem{example}[theorem]{Example}
\newtheorem*{observation*}{Observation}

\newtheorem*{Goal*}{Goal}

\newtheorem*{Assumption*}{Assumption}

\theoremstyle{remark}
\newtheorem*{remark*}{Remark}
\newtheorem{remark}[theorem]{Remark}

\numberwithin{equation}{section}

\newcommand{\ow}{\omega}

\newcommand{\lda}{\lambda}
\newcommand{\p}{\partial}

\newcommand{\C}{{\mathbb{C}}}
\newcommand{\R}{{\mathbb{R}}}

\newcommand{\Z}{{\mathbb{Z}}}
\newcommand{\N}{{\mathbb{N}}}
\newcommand{\D}{{\mathbb{D}}}

\newcommand{\w}{\wedge}

\newcommand{\uline}{\uline}

\newcommand{\PP}{\mathcal{P}}

\newcommand{\MM}{\mathcal{M}}

\newcommand{\W}{\widehat{W}}

\DeclareMathOperator{\id}{id}

\DeclareMathOperator{\Ham}{Ham}

\DeclareMathOperator{\Spec}{Spec}

\DeclareMathOperator{\Fix}{Fix}
\DeclareMathOperator{\sh}{shift}

\DeclareMathOperator{\HW}{HW}
\DeclareMathOperator{\CZ}{CZ}
\DeclareMathOperator{\RS}{RS}

\DeclareMathOperator{\HF}{HF}

\DeclareMathOperator{\CF}{CF}

\begin{document}

\title{A computation of the ring structure in wrapped Floer homology}

\author{Hanwool Bae}
\author{Myeonggi Kwon}

\thanks{The first author was partially supported by the National Research Foundation of Korea(NRF) grant funded by the Korea government(MSIT) (No.2020R1A5A1016126).}
\thanks{The second author was partially supported by the Institute for Basic Science (IBS-R003-D1) and by the SFB/TRR 191 `Symplectic Structures in Geometry, Algebra and Dynamics' funded by the DFG (Projektnummer 281071066 – TRR 191).}

\begin{abstract}
We give an explicit computation of the ring structure in wrapped Floer homology of a class of real Lagrangians in $A_k$-type Milnor fibers. In the $A_k$-type plumbing description, those Lagrangians correspond to the cotangent fibers or the diagonal Lagrangians. The main ingredient of the computation is to apply a version of the Seidel representation. For a technical reason, we first carry out computations in v-shaped wrapped Floer homology, and this in turn gives the desired ring structure via the Viterbo transfer map.
\end{abstract}

\address{Center for Quantum Structures in Modules and Spaces, Seoul National University, Seoul, South Korea}
 \email {hanwoolb@gmail.com}

\address{Department of Mathematics Education, Sunchon National University, Sunchon 57922, Republic of Korea}
 \email {mkwon@scnu.ac.kr}


\maketitle

\section{Introduction}

For an admissible Lagrangian $L$ in a Liouville domain, we consider a version of Lagrangian Floer homology, called \emph{wrapped Floer homology} $\HW_*(L)$, introduced in \cite{AbSe10}. Wrapped Floer homology is equipped with a family of algebraic structures, called $A_{\infty}$-structures, which include a ring structure with unity. Even though there are some structural results, for example \cite{AbPoSc08, EkLe17}, it is usually difficult to compute those structures explicitly as with other Floer theoretical invariants.

A \emph{Brieskorn Milnor fiber} is a (completed) Liouville domain defined as a regular level set of a Brieskorn-type complex polynomial. They have played an important role in symplectic and contact topology, especially as computational examples in Floer theory. Various Floer theoretical invariants of Brieskorn Milnor fibers are studied in \cite{EtLe17, Fa15, Gu17, KwvK16, Ue19, Ue16, Us99, vK08}.
 
  
 In this paper, we give an explicit computation of the ring structure in wrapped Floer homology of real Lagrangians in $A_k$-type Brieskorn Milnor fibers. In Section \ref{sec: defrealLag}, we define a class of real Lagrangians in $A_k$-type Milnor fibers which are of particular interest. This includes the real Lagrangians considered in \cite{KiKwLe18} where the graded group structure of wrapped Floer homology is computed. Note that $A_k$-type Milnor fibers can be seen as the $k$-fold linear plumbings of cotangent bundles $T^*S^n$, see Section \ref{sec: realplumbings}. The real Lagrangians we consider are the cotangent fibers away from the plumbing regions or the diagonal Lagrangians in the plumbing regions, see Section \ref{sec: fibersdiagonals}.

The main idea of our computations comes from  the work of Uebele \cite{Ue19} which uses an idea of the \emph{Seidel representation} to study the ring structure in symplectic homology. For a closed symplectic manifold $(M, \ow)$, Seidel \cite{Se97} introduced a representation of the fundamental group $\pi_1(\Ham(M, \ow))$ of Hamiltonian diffeomorphisms on the Hamiltonian Floer homology $\HF_*(M, \ow)$. More precisely, for a loop of Hamiltonian diffeomorphisms on $M$, say $g : S^1 \rightarrow \Ham(M, \ow)$, we have an associated group isomorphism $S_g: \HF_*(M, \ow) \rightarrow \HF_{*}(M, \ow)$ up to a degree shift. We call $S_g$ a \emph{Seidel operator}. In particular the Seidel operator satisfies the so-called \emph{module property}, namely $S_g(x \cdot y) = S_g(x) \cdot y$ for $x, y \in \HF_*(M, \ow)$. A relative version of the Seidel representation in Lagrangian Floer homology for closed symplectic manifolds was studied in \cite{HuLa10}.

We carry out this idea in the context of wrapped Floer theory. For a technical reason, we do not work directly with wrapped Floer homology, but instead we construct a Seidel operator in a v-shaped version of wrapped Floer homology. \emph{V-shaped wrapped Floer homology}, denoted by $\widecheck{\HW}_*(L)$, is a variant of wrapped Floer homology which also admits a ring structure with unity, see Section \ref{sec: vhw} for its definition. This can be regarded as a relative version of v-shaped symplectic homology introduced in \cite{CiFrOa10}. We remark that the v-shaped wrapped Floer homology is group isomorphic to the wrapped Floer homology of the trivial Lagrangian cobordisms \cite{CiOa18} and also to the Lagrangian Rabinowitz Floer homology \cite{Da18, Me14}.

An important property of v-shaped wrapped Floer homology for our purpose is that under a certain index-positivity condition on the contact type boundary, called \emph{product-index-positivity} (Definition \ref{def: prodindpos}), its ring structure can be defined purely in the symplectization part of the completion of a Liouville domain, see Corollary \ref{cor: ringindependence}. A related discussion can also be found in \cite[Section 9.5]{CiOa18}. This property  allows us to define a Seidel operator in $\widecheck{\HW}_*(L)$ by taking a path of Hamiltonian diffeomorphisms on the symplectization part which extends the Reeb flow on the contact boundary. Such an extension of the Reeb flow to the whole Liouville filling is not possible in our examples, see Remark \ref{rmk: notextend}. 

To formulate a Seidel operator in v-shaped wrapped Floer homology, let $(W, \lda)$ be a Liouville domain with a Liouville form $\lda$. Denote the contact type boundary by $(\Sigma, \xi = \ker \alpha)$ where $\alpha$ is a contact form given by $\alpha : = \lda|_{\Sigma}$. Let $L$ be an admissible Lagrangian, meaning that it is exact and intersects with $\Sigma$ in a Legendrian $\mathcal{L} : = \p L$. For simplicity of gradings in Floer homology, we put topological assumptions in \eqref{eq: topconditions} throughout this paper.

We say that the Reeb flow $\phi_{R}^t$ on $(\Sigma, \alpha)$ is \emph{$\mathcal{L}$-periodic} if there exists a positive real number $T_0$ such that $p \in \mathcal{L}$ if and only if $\phi_{R}^{T_0}(p) \in \mathcal{L}$. Then, extending the Reeb flow on $(\Sigma, \alpha)$, we can define a path of Hamiltonian diffeomorphisms $g: [0, 1] \rightarrow \Ham(\R_+ \times \Sigma)$ on the symplectization $\R_+ \times \Sigma$ of the form
$$
g_t(r, y) = (r, \phi_R^{f(t)}(y)). 
$$
Here $f$ is a smooth function on $[0,1]$ such that $g_0(r, y), g_1(r, y) \in \R_+ \times \mathcal{L}$ if and only if $(r, y) \in \R_+ \times \mathcal{L}$. We use the path $g$ to define a Seidel operator on $\widecheck{\HW}_*(L)$ in the following theorem.

\begin{theorem}\label{thm: introthm1}
Suppose that the triple $(\Sigma, \xi, \mathcal{L})$ is product-index-positive and the Reeb flow $\phi_R^t$ is $\mathcal{L}$-periodic. Then we have a graded group isomorphism
$$
S_g : \widecheck{\HW}_*(L) \rightarrow \widecheck{\HW}_{* +I(g)}(L)
$$
up to a degree shift $I(g)$ depending on $g$, and it satisfies a module property
$$
S_g(x \cdot y) = S_{g_0}(x) \cdot S_g(y) = S_g(x) \cdot S_{g_1}(y) 
$$
for $x, y \in \widecheck{\HW}_*(L)$.
\end{theorem}

The degree shift $I(g)$ in the above theorem is given by a Maslov-type index for paths of Hamiltonian diffeomorphisms, see Section \ref{sec: deg of I(g)}. For our purpose, it is crucial to compute $I(g)$ in examples. In Section \ref{sec: compu Ig}, we observe that if the triple $(\Sigma, \alpha, \mathcal{L})$ forms a \emph{real contact manifold} and the Reeb flow $\phi_R^t$ is $\mathcal{L}$-periodic, then we can describe the shift $I(g)$ in terms of the Maslov index of the principal Reeb chords, see Lemma \ref{lem: computingI(g)}. This will be sufficient to determine $I(g)$ explicitly in many cases.

\subsection*{Outline of computations} Here we outline our computation of the ring structure of (ordinary) wrapped Floer homology. Let us take the real Lagrangian $L_0$  in $A_k$-type Milnor fiber defined in Section \ref{sec: defrealLag}. This can be seen as a cotangent fiber in $A_k$-type plumbing of $T^*S^n$ as in Proposition \ref{prop: ident_diag_fiber}. We first compute the group structure on $\widecheck{\HW}_{*}(L_0)$ using a Morse--Bott spectral sequence in Section \ref{sec: compgroupvshaped}, and the result is given as follows. 
\begin{proposition}
For $n \geq 3$, we have
$$
\widecheck{\HW}_*(L_0) = \begin{cases} \Z_2 & * = I \cdot N -n +1, I \cdot N \text{ with $N \in \Z$}; \\ 0 & \text{otherwise}.  \end{cases}
$$
Here $I = (n-2)(k+1)+ 2$.
\end{proposition}

 An essential feature of the above graded group is that the generators of $\widecheck{\HW}_*(L_0)$ appear periodically due to the $\mathcal{L}$-periodicity of the Reeb flow, and the rank of $\widecheck{\HW}_*(L_0)$ in each degree is at most one. Because of this simplicity, the module property in Theorem \ref{thm: introthm1} of the Seidel operator turns out to be enough to determine the full ring structure of $\widecheck{\HW}_{*}(L)$ algebraically, except for the case when $k=1$.
 
\begin{example}
In the case when $n=3$ and $k=2$, we have
$$
\widecheck{\HW}_*(L_0) = \begin{cases} \Z_2 & * =\dots,-2, 0, 3, 5, 8, 10, \dots \\ 0 & \text{otherwise}  \end{cases}
$$
with the degree shift in Theorem \ref{thm: introthm1} given by $I(g) =5$ . For degree reasons, one finds that it suffices to determine the product of the form
$$
c_5\cdot c_k, \quad c_k \cdot c_5
$$
where $c_{k}$ denotes the non-trivial class of degree $k$. Observe that the degrees with $\widecheck{\HW}_*(L_0) \neq 0$ appear in the pattern
$$
\cdots \xrightarrow{+5}  \overbrace{-2, 0} \xrightarrow{+5} \overbrace{3, 5} \xrightarrow{+5}  \overbrace{8, 10}  \xrightarrow{+5}  \cdots
$$
which essentially comes from the periodicity of the Reeb flow. We see, in particular, that
$S_g(c_k) = c_{k+5}$
since the Seidel operator $S_g$  in Theorem \ref{thm: introthm1} is an isomorphism and $\widecheck{\HW}_*(L_0)$ is of at most rank one. Using the module property of $S_g$, we deduce that 
$$
c_5\cdot c_k = S_g(c_0) \cdot c_k = S_g(c_0 \cdot c_k) = S_g(c_k) = c_{k+5}.
$$
Here $c_0$ is the unit element. Similarly one has $c_k \cdot c_5 = c_{k+5} =c_5\cdot c_k$. From this we can conclude that
$$
\widecheck{\HW}_*(L_0) \cong \Z_2[x, y, y^{-1}]/ ( x^2)
$$
as rings with $|x| =3$ and $|y| = 5$, via $c_3 \mapsto x$ and $c_5 \mapsto y$.
\end{example}

In the general case, the computation is given as follows.
\begin{proposition}
For $n \geq 3$ and $k\geq 2$, we have a graded ring isomorphism
$$
\widecheck{\HW}_*(L_0) \cong \Z_2[x, y, y^{-1}]/ ( x^2),
$$
where $|x| = I -n +1$ and $|y| = I$ with $I = (n-2)(k+1)+2$.
\end{proposition}

Next, note that the ring structures on $\HW_*(L)$ and $\widecheck{\HW}_*(L)$ are in general related by a ring homomorphism
$$
\HW_*(L) \rightarrow \widecheck{\HW}_*(L),
$$
which is called the \emph{Viterbo transfer map}. This map is defined as a continuation map in a standard way in Floer theory, see Section \ref{sec: viterbomap}. We prove that, due to the simplicity of the graded group structure, the Viterbo transfer map is in fact injective (Theorem \ref{thm: computationoftransfermap}). We therefore obtain the ring $\HW_*(L_0)$ as a subring in $\widecheck{\HW}_*(L_0)$.
\begin{proposition}
For $n \geq 3$ and $k \geq 2$, we have a graded ring isomorphism
$$
\HW_*(L_0) \cong \Z_2[x,y]/(x^2) $$
where $|x| = I - n+1$ and $|y| = I$ with $I = (n-2)(k+1) +2$.
\end{proposition}
We remark that the wrapped Floer homology of a cotangent fiber in a cotangent bundle is computed in terms of the homology of the path space of the base manifold in \cite{AbPoSc08}. Therefore if $k=1$ which is the case when $L_0$ is a cotangent fiber in $T^*S^n$, we already know that $\HW_*(L_0) \cong \Z_2[x]$ as graded rings with $|x| = n-1$.

\subsection{Organization of the paper}

In Section 2 we describe v-shaped wrapped Floer homology and its independence of fillings provided with the notion of index-positivity of contact structures. In Section 3 we define certain real Lagrangians in $A_k$-type Milnor fibers and compute the group structure of the v-shaped wrapped Floer homology of them using a Morse--Bott spectral sequence. Section 4 is devoted to the construction of the Seidel operator by which we compute the ring structure of the v-shaped wrapped Floer homology. In the last section we conclude the computation of the ring structure of the ordinary wrapped Floer homology via a Viterbo transfer map.


\subsection*{Acknowledgements} The authors would like to thank Urs Fuchs, Jungsoo Kang, Yong-Geun Oh, and Otto van Koert for helpful discussions and comments, Chi Hong Chow for careful proofreading.
A part of this article was written during the second author's visit to the Research Institute for Mathematical Sciences at Kyoto University. He deeply thanks for their warm hospitality. 

\section{V-shaped wrapped Floer homology}\label{sec: vhw}


\subsection{V-shaped admissible Hamiltonians} Let $(W, \ow = d\lda)$ be a Liouville domain with a Liouville form $\lda$. Its boundary $\Sigma := \p W$ admits a contact form $\alpha : = \lda|_{\Sigma}$ and hence a contact structure $\xi = \ker \alpha$. Let $L$ be an exact  Lagrangian in $W$ which is \emph{admissible} in the sense that $L$ intersects $\Sigma$ transversely in a Legendrian $\mathcal{L} := L \cap \Sigma$, and the Liouville vector field $X_{\lda}$ is tangent to $L$ near the boundary. For grading purpose, see \cite[Remark 4.6]{Ri13}, we assume throughout this paper the following topological conditions.
\begin{equation}\label{eq: topconditions}
\text{$\pi_1(W, L), \pi_1(W), \pi_1(\Sigma), \pi_1(\Sigma, \mathcal{L}), c_1(TW), c_1(\xi), c_1(TW, TL), c_1(\xi, T\mathcal{L})$ vanish.}
\end{equation}

\begin{remark}
One can work with less topological conditions, but the above ones make the construction of Floer homology reasonably simpler and are all fulfilled in our examples.
\end{remark}

We denote the completions of $W$ and $L$ by 
$$
\widehat W = W \cup_{\p W} ([1, \infty) \times \Sigma), \quad \widehat L = L \cup_{\p L} ([1, \infty) \times \mathcal{L}),
$$
respectively. Denoting by $r$ the coordinate of $[1, \infty)$ in the completion, the Liouville form $\lda$ is completed by $\widehat \lda = r \alpha$ in $[1, \infty) \times \Sigma$. We denote the spectrum of Reeb chords on the contact boundary by
$$
\Spec(\Sigma, \alpha, \mathcal{L}) : = \{T \in \R \;|\; \text{$T$ is a period of a Reeb chord on $(\Sigma, \alpha)$ whose end points are in $\mathcal{L}$}\}.
$$

\begin{definition}\label{def: vshapedadmissible} 

A Hamiltonian $H :\widehat{W} \to \R$ is called \emph{v-shaped admissible} if
\begin{enumerate}
\item Hamiltonian 1-chords with end points in $L$ are non-degenerate;
\item $H \leq 0$ in a small neighborhood of the boundary $\partial W = \{ r=1\}$, and $H \geq 0$ elsewhere;
\item $H(r, p) = h(r)$ for a convex function $h$ on $\R_+$ such that $h(r) =\mathfrak{a} r+ \mathfrak{b}$ for $r \gg 1$ where $\mathfrak{a} \not \in \Spec(\Sigma, \alpha, \mathcal{L})$, $\mathfrak{b} \in \R$.
\end{enumerate}

We denote the set of Hamiltonian $1$-chords of $H$ in $\widehat W$ relative to $\widehat L$ by $\mathcal{P}_{\widehat L}(H)$.
\end{definition}

\begin{remark}
The first condition can always be achieved by perturbing the Lagrangian $\widehat L$, see the proof in \cite[Lemma 4.3]{Ri13}. The main difference from the definition of the usual admissible Hamiltonians for ordinary (non-v-shaped) wrapped Floer homology is the second condition.
\end{remark}

\begin{remark}
We use the convention that $\ow(X_H, \cdot) = -dH.$
\end{remark}

\subsection{Floer chain complex} \label{sec: Floerchaincomplex} Since we have assumed $\pi_1(W, L) = 0$, every Hamiltonian 1-chord relative to $\widehat L$, i.e. the end points are in $\widehat L$, is contractible. Let $x:[0, 1] \rightarrow \widehat W$ be a Hamiltonian 1-chord. By choosing a capping half disk of $x$, we obtain a symplectic trivialization $\tau_x : x^*T\widehat W \rightarrow [0, 1]\times \R^{2n}$ along $x$ such that $\tau_x(T_{x(t)}\widehat L) = t \times \Lambda_0$ where $\Lambda_0$ denotes the horizontal Lagrangian in $\R^{2n}$. We associate a path of symplectic matrices $\Phi_x: [0, 1] \rightarrow Sp(2n)$ to $x$ by
$$
\Phi_x(t) : = \tau_x(t) \circ d\phi_H^t(x(0)) \circ \tau_x(0)^{-1}.
$$
The \emph{Maslov index} of $x$ is then defined by
$$
\mu(x) : = \mu_{\RS}(\Phi_x\Lambda_0, \Lambda_0)
$$
where $\mu_{\RS}$ denotes the Robbin--Salamon index for paths of Lagrangians, defined in \cite{RoSa93}. Under the topological assumptions \eqref{eq: topconditions}, the Maslov index of $x$ does not depend on the choice of capping half disks. See \cite[Section 2.1]{KiKwLe18} for more details. 

For a v-shaped admissible Hamiltonian $H$ and a time-dependent family of admissible (i.e. contact type at the cylindrical end) almost complex structures $J = \{J_t\}$ on $\widehat W$, we define the filtered Floer chain complex $(\CF_*^{<c}(L;H, J), \p)$ with $c$ as follows. Define a Hamiltonian action functional on the path space of $\widehat W$ relative to $\widehat L$ by
\begin{equation}\label{eq: Hamactfunl}
\mathcal{A}_H (x) : =- g_{\widehat L}(x(1)) + g_{\widehat L}(x(0)) + \int_0^1 x^*\widehat \lda - \int_0^1 H(x(t)) dt
\end{equation}
where $\lda|_{\widehat L} = dg_{\widehat L}$ for some smooth function $g_{\widehat L}$. Then the critical points of $\mathcal{A}_H$ are exactly the elements of $\mathcal{P}_{\widehat L}(H)$. The chain group $\CF_*^{<c}(L;H, J)$ is the $\Z_2$-vector space generated by Hamiltonian 1-chords relative to $\widehat L$ with action less than $c$, in other words,
$$
\CF_*^{<c}(L;H, J) := \bigoplus_{\substack{x \in \mathcal{P}_L(H)\\ \mathcal{A}_H(x) < c}} \Z_2\langle x \rangle.
$$
We grade the chain complex by
$$
|x| : = \mu(x) - \frac{n}{2}.
$$

The differential $\p$ of the Floer chain complex counts the number of a moduli space of Floer strips. We define 
\begin{eqnarray*}
\widetilde{\mathcal{M}}(x_-, x_+;H, J): =\{u:\R\times [0,1]\to \widehat{W} &|&  \p_s u+J_t(\p_t u -X_H(u))=0,  \\
&& \lim_{s\to \pm \infty}u(s,t)=x_{\pm}(t),\\
&& u(s,0),u(s,1)\in \widehat{L}\}.
\end{eqnarray*}
For $x_- \neq x_+ \in \mathcal{P}_{\widehat L}(H)$, we have a free $\R$-action on $\widetilde{\MM}(x_{-},x_{+} ;H,J)$ given by translation with respect to the $s$-parameter. Denote the quotient space by
$$
{\MM}(x_{-},x_{+} ;H,J) : = \widetilde{\MM}(x_{-},x_{+} ;H,J)/\R.
$$
For generic $J$, the moduli space ${\MM}(x_{-},x_{+} ;H,J)$ is a smooth manifold of dimension $|x_+| - |x_-|-1$. The differential $\p: \CF_k^{<c}(L;H, J) \rightarrow \CF_{k}^{<c}(L;H, J)$ is defined by counting the elements of ${\MM}(x_{-},x_{+} ;H,J)$ modulo $2$ as
$$
\p(x_+) = \sum_{\substack{x_+ \in \PP_{\widehat L}(H) \\ |x_-| = |x_+|-1}} \#_{2} {\MM}(x_{-},x_{+} ;H,J) \cdot x_-.
$$

We define the chain complex of action window $(a, b)$, for $a < b \in \R$, by
$$
\CF_*^{(a, b)}(L;H, J) : = \CF_*^{<b}(L;H, J)/\CF_*^{<a}(L;H, J)
$$
with the induced differential by $\p$. Then the filtered Floer homology $\HF_*^{(a, b)}(L;H, J)$ is defined to be the homology of the chain complex $(\CF_*^{(a, b)}(L;H, J), \p)$.


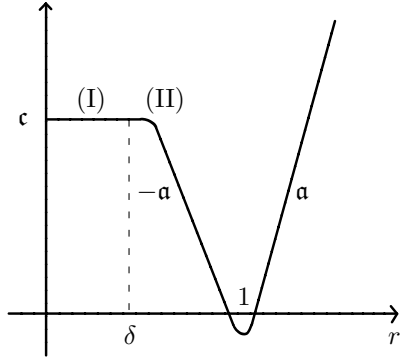
\begin{figure}[htb]
	\centering
	\begin{overpic}[width=150pt,clip]{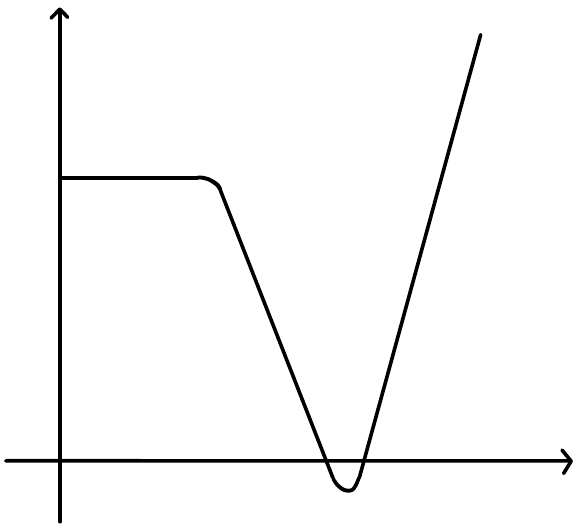} 
	\dashline{3.0}(47,18)(47,90)
	\put(145,5){$r$}
	\put(27,95){(I)}
	\put(53,95){(II)}
	\put(88,20){$1$}
	\put(45,5){$\delta$}
	\put(50,60){$\mathfrak{-a}$}
	\put(110,60){$\mathfrak{a}$}
	\put(5,88){$\mathfrak{c}$}
	\end{overpic}
	\caption{V-shaped Hamiltonian}
	\label{fig: vshapedHam}
\end{figure}

\subsection{V-shaped wrapped Floer homology} \label{sec: continuationamap} For two v-shaped admissible Hamiltonians with $H_+ \leq H_-$ and generic almost complex structures $J_+$ and $J_-$ (not necessarily distinct), 
we have a continuation map
\begin{equation}\label{eq: continuation}
f_{H_{+},H_{-}} : \HF_*^{(a, b)}(L;H_+, J_+) \rightarrow \HF_*^{(a, b)}(L;H_-, J_-).
\end{equation}
Let us briefly review its construction. We take a one parameter family $\{ (H_s,J_s)\}_{s \in \R}$ of Hamiltonians and almost complex structures such that
\begin{itemize}
\item $
(H_s,J_s) = \begin{cases} (H_{-},J_{-}) & s\leq -R \\ (H_+,J_+) & s\geq R \end{cases}
$
 for some sufficiently large $R >0$;
 \item $H_s$ is non-increasing with respect to $s$, i.e. $\frac{\p H_s}{\p s} \leq 0$.
 \end{itemize}
We define the moduli space $\mathcal{M} ( x_-, x_+ ; H_s, J_s)$ of parametrized Floer solutions by 
\begin{eqnarray*}
 \mathcal{M} ( x_{-}, x_+ ; H_s, J_s) = \{ u : \R \times [0,1] \to \widehat{W} &|&  \p_s u+J_{s,t}(\p_t u -X_{H_s} (u))=0,  \\
& & \lim_{s\to \pm \infty} u(s,t)=x_{\pm}(t),\\
& & u(s,0),u(s,1)\in \widehat{L} \}
\end{eqnarray*}
for $x_{-} \in \PP_{\widehat{L}} (H_-)$ and $x_+ \in \PP_{\widehat{L}}(H_+)$. For generic homotopies $\{(H_s,J_s)\}$, the moduli space is a smooth manifold of dimension $|x_+| - |x_-|$. The chain level continuation map $f_{H_+,H_-} : \CF_*^{(a,b)} (L;H_+,J_+) \to \CF_*^{(a,b)} (L;H_-,J_-) $ is then defined by 
$$ 
f_{H_+,H_{-}} (x_+) = \sum_{ \substack{x_{-} \in \PP_{\widehat{L}} (H_{-})\\ \mu(x_{-}) = \mu(x_{+})}} \#_{2} \MM(x_-,x_+;H_s,J_s) \cdot x_{-}.
$$ 
A standard compactness result in Floer theory, with the condition that $H_s$ is non-increasing with respect to $s$, shows that $f_{H_+,H_{-}}$ is a chain map. The induced map on the homology is called a continuation map, which we denote still by $f_{H_+,H_-}$. This yields a direct system on Floer homology $\HF_*^{(a, b)}(L;H, J)$ over the set of pairs of a v-shaped admissible Hamiltonian $H$ and an almost complex structure $J$.

For a given action window $(a, b)$, we can take a cofinal family of v-shaped admissible Hamiltonians such that $H = h(r)$ is of the form in Figure \ref{fig: vshapedHam}. We can choose the constants $\delta$, $\mathfrak{c}$ so that Hamiltonian 1-chords in the region corresponding to (I) and (II) in Figure \ref{fig: vshapedHam} are not in the action window; see \cite[Section 3.2.2]{Da18} and the proof in \cite[Proposition 2.9]{CiFrOa10} for more details. Over such a cofinal family, we define the \emph{v-shaped wrapped Floer homology of $L$ of action $(a, b)$} to be the direct limit 
$$
\widecheck{\HW}^{(a, b)}_*(L) : = \varinjlim_{\frak{a}} \HF_*^{(a, b)}(L;H, J).
$$ 
We define the \emph{(full) v-shaped wrapped Floer homology of $L$} by
$$
\widecheck{\HW}_*(L) : = \varinjlim_{b} \varprojlim_a \widecheck{\HW}^{(a, b)}_*(L).
$$

\begin{remark}
As a group the v-shaped wrapped Floer homology is isomorphic to the Lagrangian Floer homology for the trivial Lagrangian cobordisms defined in \cite{CiOa18} (up to minor conventional differences). Moreover, in \cite[Section 3.3.2]{Da18}, it is shown that the v-shaped wrapped Floer homology is group isomorphic to the Lagrangian Rabinowitz Floer homology. This is parallel to the fact in \cite{CiFrOa10} that the v-shaped symplectic homology is isomorphic to the Rabinowitz Floer homology.
\end{remark}

\subsection{Ring structure on $\widecheck{\HW}(L)$}
In this section, we outline a construction of the ring structure of the v-shaped wrapped Floer homology $\widecheck{\HW}_*(L)$. We first define a half-pair-of-pants product
$$
\HF(L;H_1,J_1) \otimes \HF(L;H_2,J_2) \to \HF(L;H_0,J_0),
$$
for pairs $(H_i,J_i)$, $i=0,1,2$, of a v-shaped admissible Hamiltonian and a time-dependent almost complex structure.

Let $S$ be a two dimensional disk with three boundary points $z_0,z_1,z_2 \in \p S$ removed, and let $j$ be a complex structure on $S$.
To introduce the notion of \emph{strip-like ends} near the boundary punctures, let us consider the semi-infinite strips
$$ Z^{\pm} = \R^{\pm} \times [0,1],$$
which is equipped with the standard complex structure, namely $j \frac{\p}{\p s} = \frac{\p}{\p t}$. 

We consider a \emph{positive} strip-like end near $z_i$ for $i=1,2$, that is, a holomorphic embedding
$$\epsilon_i : \R^+ \times [0,1] \to S$$
such that
$$ \epsilon_i^{-1} (\p S) = \R^+ \times \{0,1\} \text{ and }  \lim_{s \to \infty} \epsilon_{i}(s, \cdot) = z_i.$$ 
We also consider a \emph{negative} strip-like end near $z_0$ in a similar way. We further require that the images of $\epsilon_i$ are pairwise disjoint.

We choose the following data to define the moduli space of half-pair-of-pants.
\begin{itemize}
\item A family of almost complex structures $J^S$ parametrized by $S$, which is of contact type at the cylindrical ends and satisfies $J^S_{\epsilon_i(s,t)} =J_i^t$ for each $i=0,1,2$.

\item A 1-form $\beta \in \Omega^1 (S)$ such that $\epsilon_i^* \beta = dt$ and its restriction to the boundary $\p S$ is zero.

\item A family of Hamiltonians $H^S = \{ H^S(z, \cdot ) \in C^{\infty} (\widehat{W})\}_{z \in S}$ parametrized by $S$ such that $H^S (z, \cdot)$ is a v-shaped admissible Hamiltonian on $\widehat{W}$ for each $z\in S$ and that $H^S( \epsilon_i(s,t),\cdot)  = H_i^t(\cdot)$ for each $i=0,1,2$ and $ d(H^S(\cdot, w) \beta) \leq 0 $ for all $w \in \widehat{W}$.
(The last property is necessary to apply the maximum principle.)

\end{itemize}

Applying Stokes' theorem to $H^S (\cdot, w) \beta$ for some $w \in \widehat{W}$, we observe that the sum of $H_1$ and $H_2$ is less than or equal to $H_0$. Let us assume the Hamiltonians $H_0,H_1, H_2$ satisfy this condition from now on.

\begin{definition}\label{def: ringFloereq}
Let $x_i \in \PP_{\widehat L}(H_i)$ be a Hamiltonian chord for $i=0,1,2$. The \emph{moduli space of half-pair-of-pants} $\MM(x_0, x_1, x_2;\beta,H^S,J^S)$ consists of all maps $u: S \to \W$ of finite energy, which satisfy the following conditions.
\begin{equation*}\label{eq : asymptotic}
\begin{cases}
u (\p S) \subset \widehat{L},\\
\displaystyle{ \lim_{s \to -\infty} u(\epsilon_0(s,t))}=x_0(t),\\
\displaystyle{\lim_{s \to +\infty} u(\epsilon_i(s,t))}= x_i(t) \text{ for } i=1,2,\\
(du - X_{H^S} \otimes \beta)^{0,1} := \frac{1}{2}\left( ( du - X_{H^S} \otimes \beta) + J^S \circ (du - X_{H^S} \otimes \beta) \circ j \right) =0,
\end{cases}
\end{equation*}
where $X_{H^S}$ is the Hamiltonian vector field associated to $H^S$.
\end{definition}

For a generic pair $(H^S, J^S)$, the moduli space $\MM(x_0, x_1, x_2;\beta,H^S,J^S)$ is a smooth manifold, and its dimension is equal to
$$
\dim \MM(x_0, x_1, x_2;\beta,H^S,J^S) = \mu(x_1) + \mu(x_2) - \mu(x_0) -\frac{n}{2}.
$$
As a consequence, the moduli space is zero-dimensional if $|x_0| = |x_1| + |x_2|$. We define the half-pair-of-pants product by
$$ x_1 \cdot x_2  = \sum_{\substack{x_0 \in \PP_L(H_0)\\ |x_0| = |x_1| + |x_2|}} \#_2 {\MM(x_0, x_1, x_2;\beta,H^S,J^S)} \cdot x_0$$
for $x_i \in \PP_{\widehat L}(H_i), i=1,2$ and extend this $\Z_2$-linearly to Floer chain complexes.

The product structure descends to homology as it satisfies the Leibniz rule. Arguing as in \cite[Section 10]{CiOa18}, we can show that the product structure induces a map
$$ \HF^{(a_1,b_1)} (L;H_1,J_1) \otimes \HF^{(a_2,b_2)} (L;H_2,J_2) \to \HF^{(a_0,b_0)}(L;H_0,J_0),$$
for any $a_i<b_i$, $i=0,1,2$ such that $a_i < 0 <b_i$ for $i=1,2$, $a_0 = \text{max} \{ a_1+b_2,a_2+b_1\}$ and $b_0 = b_1+b_2$.

Since the product structure is compatible with the continuation map, we get a product structure on the v-shaped wrapped Floer homology by first taking the inverse limit as $a_1,a_2 \to -\infty$ and then taking the direct limit as $b_1,b_2 \to \infty$. Furthermore, the product structure is associative and carries a strict unit. We refer the readers to \cite{Ab15} and \cite{Ri13} for more details. In summary, the v-shaped wrapped Floer homology has a unital associative ring structure.



\begin{remark}
The ring $\widecheck{\HW}_* (L)$, as well as $\HW_*(L)$, may not be commutative, see \cite[Section 6]{Ri13}.
\end{remark}

\subsection{Independence of fillings} \label{sec: indfill} In this section, we show that the ring structure of the v-shaped wrapped Floer homology $\widecheck{\HW}_*(L)$ can be defined purely in the symplectization part $\R_+ \times \p W$ of the completion $\widehat W$, provided with an index-positivity condition in Definition \ref{def: prodindpos}

\subsubsection{Stretching the neck} \label{sec: stretchingtheneck}
The key technique to show that $\widecheck{\HW}_*(L)$ can be defined in the symplectization part (independently of the filling) is the \emph{stretching-the-neck operation} in (relative) symplectic field theory. In this section, we briefly describe this operation following the idea in \cite[Lemma 2.4]{CiOa18}; see also \cite[Section 2.6]{ChCoRi19}.

Let $H: \widehat W \rightarrow \R$ be a v-shaped admissible Hamiltonian.  Take a small neighborhood $I_{\delta_0} \times \Sigma$ of the contact hypersurface $\{\delta_0\} \times \Sigma$ in $\widehat W$ where $\delta_0 > 0$ is sufficiently small and $I_{\delta_0}$ is a closed interval containing $\delta_0$. Since $H$ is v-shaped, we may assume that $H$ is constant in the neighborhood $I_{\delta_0} \times \Sigma$. Let $J$ be a cylindrical or equivalently SFT-like almost complex structure on $\widehat W$, meaning that
\begin{itemize}
\item $J$ preserves the contact structure;
\item $J$ is invariant under the Liouville direction in the neighborhood $I_{\delta_0} \times \Sigma$;
\item $J$ sends the Liouville vector field $r\p_r$ to the Reeb vector field $R_{\alpha}$.
\end{itemize}

For a deformation parameter $R > 0$, we pick a diffeomorphism $\phi_R$ from the interval $[-R, R]$ to the interval $I_{\delta_0}$ with $\phi_{R}'(\pm R) = 1$. This induces a diffeomorphism $\phi_R \times \id: [-R, R]\times \Sigma \rightarrow I_{\delta_0}\times \Sigma$. Let $J_0$ be a compatible almost complex structure on $[-R, R]\times \Sigma$ which is SFT-like in the above sense and such that $J_0|_{\{0\}\times \Sigma}$ coincides with $J|_{\{\delta_0\}\times \Sigma}$. We define an almost complex structure $J_R$ on the completion $\widehat W$ by
$$
J_R = \begin{cases} (\phi_R \times \id)_*J_0 & \text{on $I_{\delta_0} \times \Sigma$};\\ J &  \text{elsewhere.}  \end{cases}
$$
By the condition that $\phi_{R}'(\pm R) = 1$ and that $J_0|_{\{0\}\times \Sigma}$ coincides with $J|_{\{\delta_0\}\times \Sigma}$, the almost complex structure $J_R$ is well-defined. By stretching-the-neck operation we shall mean that we replace the given almost complex structure $J$ on $\widehat W$ by $J_R$, typically for an arbitrarily large $R > 0$. 

\subsubsection{Index-positivity} \label{sec: dimofmoduli} To motivate the notion of index-positivity of a contact manifold with a Legendrian, we first state the virtual dimension of moduli spaces which we shall consider later. As always in this paper, we assume the topological conditions \eqref{eq: topconditions} to ensure that all periodic Reeb orbits and chords are contractible and have well-defined indices. In addition, all of those are assumed to be non-degenerate.

Let $\chi^- = (x^-_1, \dots, x_{k_-}^-)$ and $\chi^+ = (x^+_1, \dots, x_{k_+}^+)$ be collections of Hamiltonian 1-chords in $\R_+ \times \Sigma$ relative to $L =  \R_+ \times \mathcal{L}$, and let $\Gamma = (\gamma_1, \dots, \gamma_{\ell})$ be a collection of contractible periodic Reeb orbits in $\Sigma$, and let $C = \{c_1, \dots, c_{m}\}$ be a collection of contractible Reeb chords in $\Sigma$ relative to $\mathcal{L}$. Consider the punctured disk 
$$
\mathcal{S} : = \D\setminus\{z_1^-, \dots, z_{k_-}^-, z_1^+, \dots, z_{k_+}^+, w_1, \dots, w_{\ell}, \overline w_1, \dots, \overline w_{m}\}
$$
where $z_1^-, \dots, z_{k_-}^-, z_1^+, \dots, z_{k_+}^+, \overline w_1, \dots, \overline w_{m} \in \p\D$ are boundary punctures ($z_j^+$ is positive, and $z_j^-$ and $w_j$ is negative for each $j$) and $w_1, \dots, w_{\ell} \in \text{int}(\D)$ are interior negative punctures. Take Floer data $(H, J, \beta)$ parametrized by $\mathcal{S}$ defined analogously to the case of the half-pair-of-pants.

Define the moduli space $\mathcal{M}(\chi^-, \chi^+, \Gamma, C; \beta, H, J)$ of maps
$$
u: \mathcal{S} \rightarrow \R_+ \times \Sigma
$$
such that $u$ satisfies the Floer equation (analogous to the equation in Definition \ref{def: ringFloereq}) and the Lagrangian boundary condition, and converges to the Hamiltonian 1-chord $x_j^{\pm}$ at $z_j^{\pm}$ in the sense of Hamiltonian Floer theory and converges to the periodic Reeb orbit $0 \times \gamma_j$ at $w_j$ in the sense of SFT and converges to the Reeb chord $0 \times c_i$ at $\overline w_i$ in the sense of (relative) SFT. 

The following formula, which is well-known, can be obtained by a standard way using the additive property of indices, together with virtual dimension formulas of moduli spaces given in \cite[Section 2.2]{Ek12}. Below, the Maslov index $\mu(c)$ of a Reeb chord $c$ is the one defined in \cite[Section 2.2.1]{KiKwLe18}; see also the proof of Lemma \ref{lem: computingI(g)}. 

\begin{theorem}\label{thm: virtualdim}
The virtual dimension of the moduli space $\mathcal{M}(\chi^-, \chi^+, \Gamma, C; \beta, H, J)$ is given by
\begin{equation}\label{eq: modulidimform}
\sum_{j=1}^{k_+}\mu(x_j^+) - \sum_{j=1}^{k_-}\mu(x_j^-) + \frac{n}{2}(2-k_+-k_-)-\sum_{j=1}^{\ell} (\mu_{\CZ}(\gamma_j)+n-3) -\sum_{j=1}^m \left(\mu(c_j) + \frac{n-3}{2}\right).
\end{equation}
\end{theorem}

\begin{remark}\label{rem: Ekholmdeg}
For a given Reeb chord $c$, Ekholm \cite[Section 2.2]{Ek12} introduced a degree $|c|$ of $c$, which is equal to the dimension of the moduli space of holomorphic curves from a disk with one negative puncture on the boundary that converge to the Reeb chord $c$ at the negative puncture.
The relation between our Maslov index $\mu(c)$ and the degree $|c|$ is given by
$$ |c| = \mu(c) +\frac{n-3}{2}.$$
\end{remark}

Motivated by the virtual dimension in Theorem \ref{thm: virtualdim} we define a notion of index-positivity as follows. The same definition, modulo conventional difference as in Remark \ref{rem: Ekholmdeg}, can be found in \cite[Section 9.5]{CiOa18}.

\begin{definition} \label{def: indexpositivity}
The triple $(\Sigma, \xi, \mathcal{L})$ of a contact manifold $(\Sigma, \xi)$ with a Legendrian $\mathcal{L}$ is called \emph{index-positive} if
there exists a contact form $\alpha$ for $(\Sigma, \xi)$ such that
\begin{itemize}
\item for every periodic Reeb orbit $\gamma$ in $(\Sigma, \alpha)$,  the Conley--Zehnder index $\mu_{\CZ}(\gamma)$ is greater than $3-n$;
\item for every Reeb chord $c$ in $(\Sigma, \alpha, \mathcal{L})$, the Maslov index $\mu(c)$ is greater than $\frac{3-n}{2}$.
\end{itemize}
\end{definition}
We note that if $(\Sigma, \xi, \mathcal{L})$ is index-positive, then the last two terms in \eqref{eq: modulidimform} are negative for periodic Reeb orbits and Reeb chords in Definition \ref{def: indexpositivity}.


\subsubsection{Independence of group structure} \sloppy Let $x_-$ and $x_+$ be Hamiltonian 1-chords of a v-shaped admissible Hamiltonian $H$, with $|x_+| - |x_-| = 1$, which are generators of the Floer chain complex $\CF_*^{(a, b)}(L;H, J)$. Denote the moduli space of Floer strips from $x_+$ to $x_-$ \emph{in the symplectization part} $\R_+ \times \Sigma$ by $\mathcal{M}^{\R_+ \times \Sigma}(x_-, x_+; H, J)$. To stress the difference, we denote the usual moduli space by $\mathcal{M}^{\widehat W}(x_-, x_+; H, J)$.

\begin{proposition}\label{prop: groupindependence}
If $(\Sigma, \xi, \mathcal{L})$ is index-positive, then 
$$
\mathcal{M}^{\widehat W}(x_-, x_+; H, J) = \mathcal{M}^{\R_+ \times \Sigma}(x_-, x_+; H, J).
$$
In other words, all Floer strips in $\widehat{W}$ with asymptotics in the symplectization part $\R_+ \times \Sigma$ are actually contained in $\R_+ \times \Sigma$.
\end{proposition}

\begin{proof}
Suppose, on the contrary, that there exists a Floer strip $u: \R \times [0, 1] \rightarrow \widehat W$ which is not entirely contained in $\R_+ \times \Sigma$. In particular $u$ intersects nontrivially with a neighborhood $I_{\delta_0} \times \Sigma \subset \R_+ \times \Sigma$ considered in Section \ref{sec: stretchingtheneck}. Then the stretching the neck operation in Section \ref{sec: stretchingtheneck} produces a sequence of Floer strips $u_{R_k}: \R \times [0, 1] \rightarrow \widehat W$, with $R_k \rightarrow \infty$ as $k \rightarrow \infty$, which is  $J_{R_k}$-holomorphic (note that $H$ is constant in $I_{\delta_0}\times \Sigma$) in the neighborhood $I_{\delta_0}\times \Sigma$. By the relative SFT compactness as in \cite[Section 11.3]{BoElHoWyZe03}, $u_{R_k}$ converges to a (partially-)holomorphic building, and in particular the top piece in the limit is an element of a moduli space of the form $\mathcal{M}(x_-, x_+, \gamma_1, \dots, \gamma_{\ell}, c_1, \dots, c_m;\beta, H, J)$ in Section \ref{sec: dimofmoduli}. 


By Theorem \ref{thm: virtualdim}, the virtual dimension of the moduli space, after modding out the $\R$-action, is 
$$
\mu(x_+) -\mu(x_-) - 1 -\sum_{j=1}^{\ell} (\mu_{\CZ}(\gamma_j)+n-3) -\sum_{j=1}^m \left(\mu(c_j) + \frac{n-3}{2}\right).
$$
Since we have assumed $|x_+| - |x_-| = 1$, the dimension equals to
$$
-\sum_{j=1}^{\ell} (\mu_{\CZ}(\gamma_j)+n-3) -\sum_{j=1}^m \left(\mu(c_j) + \frac{n-3}{2}\right).
$$
By the index-positivity, it follows that the virtual dimension is necessarily negative. This gives a contradiction.
\end{proof}

It is straightforward to see that the same argument as the proof of Proposition \ref{prop: groupindependence} also works for the case of parametrized Floer strips which we count to define continuation maps. Together with Remark \ref{rem: topocondindeindipen}, we  conclude that v-shaped wrapped Floer homology, as a graded group, can purely be defined in the symplectization part $\R_+ \times \Sigma$.

\begin{corollary}
Under the index-positivity condition, the group $\widecheck{\HW}_*(L)$ does not depend on the filling $W$ and is defined in $\R_+ \times \Sigma \in \widehat W$.
\end{corollary}

\begin{remark}\label{rem: topocondindeindipen}
By the topological assumptions in \eqref{eq: topconditions}, the Maslov index of a Hamiltonian 1-chord can be defined in terms of the Maslov index of the corresponding Reeb chord. Therefore we do not need a filling to grade (v-shaped) wrapped Floer homology.
\end{remark}

\subsubsection{Independence of ring structure} 

Under the following a bit stronger index-positivity condition, we show that the ring structure on the v-shaped wrapped Floer homology is independent of the filling. 

\begin{definition}\label{def: prodindpos}
The triple $(\Sigma, \xi, \mathcal{L})$ of a contact manifold $(\Sigma, \xi)$ with a Legendrian $\mathcal{L}$ is called \emph{product-index-positive} if there exists a contact form $\alpha$ for $(\Sigma, \xi)$ such that
\begin{itemize}
\item for every periodic Reeb orbit $\gamma$ in $(\Sigma, \alpha)$, the Conley--Zehnder index $\mu_{\CZ}(\gamma)$ is greater than $3-n$;
\item for every Reeb chord $c$ in $(\Sigma, \alpha, \mathcal{L})$, the Maslov index $\mu(c)$ is greater than $\frac{3}{2}$.
\end{itemize}
\end{definition}

Let $x_0,x_1,x_2$ be Hamiltonian 1-chords in $\R_+ \times \Sigma$ relative to $L = \R_+ \times \mathcal{L}$. Consider the moduli space $\mathcal{M}(x_0, x_1,x_2; \beta, H^S, J^S)$ of half-pair-of-pants.

\begin{proposition} If $(\Sigma,\xi, \mathcal{L})$ is product-index-positive, then
$$\mathcal{M}^{\widehat{W}}(x_0, x_1,x_2; \beta, H^S, J^S)  =\mathcal{M}^{\R_+ \times \Sigma} (x_0, x_1,x_2; \beta, H^S, J^S).$$
In other words, all half-pair-of-pants in $\widehat{W}$ with asymptotics in the symplectization part $\R_+ \times \Sigma$ lie in the symplectization.
\end{proposition}
\begin{proof}
As in the proof of Proposition \ref{prop: groupindependence}, suppose on the contrary that there exists a half-pair-of-pants $u$ in $\mathcal{M}(x_0,x_1,x_2; \beta,H^S,J^S)$ that enters the filling where $x_0,x_1,x_2$ are Hamiltonian chords with $\mu(x_1) + \mu(x_2) -\mu(x_0) - \frac{n}{2}=0$. Using the neck stretching operation again, we get a sequence of half-pair-of-pants, and by the relative SFT compactness as in \cite[Section 11.3]{BoElHoWyZe03}, the sequence converges to a holomorphic building. In contrast to the Floer strip case, the top piece of the holomorphic building can be \emph{disconnected}. Below we argue by separating cases.

We first consider the case when the top piece of the holomorphic building is connected. Then it must be an element of the moduli space of the form $\MM(x_0,x_1,x_2, \gamma_1, \dots, \gamma_l, c_1, \dots, c_m; \beta, H,J)$. The virtual dimension of this moduli space is
$$ \mu(x_1) + \mu(x_2) - \mu(x_1) -\frac{n}{2} - \sum_{j=1}^{\ell} (\mu_{\CZ}(\gamma_j)+n-3) -\sum_{j=1}^m \left(\mu(c_j) + \frac{n-3}{2}\right). $$
Due to the assumption $\mu(x_1) + \mu(x_2) - \mu(x_0) -\frac{n}{2} =0$, the dimension is equal to 
$$
- \sum_{j=1}^{\ell} (\mu_{\CZ}(\gamma_j)+n-3) -\sum_{j=1}^m \left(\mu(c_j) + \frac{n-3}{2}\right),
$$ 
which is negative by product-index-positivity. This gives a contradiction.

The remaining case happens only when there are exactly two connected components in the top piece. Indeed, every component of the top piece must contain at least one of $x_1$ and $x_2$ due to the maximum principle. Therefore we now assume that the top piece has exactly two components and let $u_1$ and $u_2$ be the maps corresponding to those two components. Without loss of generality, we may assume that $u_1$ belongs to $\MM(x_0,x_1,\Gamma_1, C_1; \beta, H,J)$ and $u_2$ belongs to $\MM(\emptyset,x_2,\Gamma_2, C_2; \beta, H,J)$ for some sets $\Gamma_j$ of periodic Reeb orbits and $C_j$ of Reeb chords, $j=1,2$.
 
Consider the case when $\Gamma = \emptyset = \Gamma_2$ and $|C_1| = 1 =|C_2| $. All the other cases can be treated in a similar way. Write $C_1= \{ c_1 \}$ and $C_2 = \{ c_2\}$. Considering that the moduli space $\MM(x_0,x_1,\Gamma_1, C_1; \beta, H,J)$ has nonnegative dimension, we get
$$ \mu(x_1) - \mu(x_0) - \mu ( c_1) - \frac{n-3}{2} \geq 0.$$
Similarly, regarding $\MM(\emptyset,x_2,\Gamma_2, C_2; \beta, H,J)$, we get
$$ \mu(x_2) +\frac{n}{2} - \mu (c_2) - \frac{n-3}{2} \geq 0.$$
Plugging the equality $\mu(x_1) -\mu(x_0) = \frac{n}{2} -\mu(x_2)$ into the first inequality, we get
two inequalities
\begin{align*} - \mu(x_2) - \mu (c_1) + \frac{3}{2} \geq 0 \\ \mu(x_2) - \mu (c_2) +\frac{3}{2} \geq 0. \end{align*}
By the product-index-positivity assumption, the first inequality leads to $\mu(x_2) <0$ while the second one leads to $\mu(x_2) >0$. This is a contradiction.
\end{proof}

\begin{corollary}\label{cor: ringindependence}
Under the product-index-positivity, the ring $\widecheck{\HW}_*(L)$ does not depend on the filling $W$ and is defined purely in $\R_+ \times \Sigma \in \widehat W$.
\end{corollary}

\begin{remark}\label{rem: indposMB}
In the case of a \emph{Morse--Bott type} contact form $\alpha$, the corresponding product-index-positivity 
requires periodic Reeb orbits $\gamma$ and chords $c$ to satisfy that
$$
\mu_{\CZ}(\gamma) > 3-n + \frac{1}{2}\dim S_{\gamma}, \quad
\mu(c) > \frac{3}{2} + \frac{1}{2} \dim S_c
$$
where $S_{\gamma}$ and $S_{c}$ denote the connected component of the Morse--Bott orbit space containing the periodic Reeb orbit $\gamma$ and the chord $c$ respectively. This follows from a well-known analogue of the virtual dimension \eqref{eq: modulidimform} in Morse--Bott setup; see e.g. \cite[Corollary 5.4]{Bo02}.
\end{remark}

\section{Real Lagrangians in $A_k$-type Milnor fibers}
\label{sec: realAkfiber}

\subsection{Lefschetz thimbles}\label{sec: Lefthim} For $k \geq 1$, let $V_k$ be the \emph{(completed) $A_k$-type Milnor fiber} of dimension $2n$ defined by
$$
V_k = \{z \in \C^{n+1} \;|\; z_0^{k+1} + z_1^2 + \cdots +z_n^2 = 1\}.
$$
The Milnor fiber $V_k$ admits an explicit Lefschetz fibration structure given by
\begin{equation}\label{eq: leffibak}
\pi : V_k \rightarrow \C, \quad z \mapsto z_0.
\end{equation}
Note that $\pi$ has exactly $k+1$ critical values; the $(k+1)$-th roots of unity, say $\xi_0 = 1, \xi_1 = e^{\frac{2\pi i}{k+1}}, \dots, \xi_{k} = e^{\frac{2k\pi i}{k+1}}$. We can associate a \emph{Lefschetz thimble} to each critical value $\xi_j$ in the following way: Consider the half straight line $\Gamma_j: [1, \infty) \rightarrow \C$ starting from $\xi_j$ given by $\Gamma_j(t) = \xi_j t$. For $t=1$, the fiber $\pi^{-1} (\Gamma_j(t))$ is singular and can be identified with $T^*S^{n-1}$ with the zero section collapsed to a point. For $t \neq 1$, the value $\Gamma_j(t)$ is a regular value of $\pi$, and the fiber $\pi^{-1} (\Gamma_j(t))$  is identified with $T^*S^{n-1}$. 

For later use, we describe the identification $\pi^{-1} (\Gamma_j(t)) = T^{*}S^{n-1}$ in more detail. Note that the fiber is written by 
$$
\pi^{-1}(\Gamma(t)) = \{z \in \C^n \;|\; z_1^2 + \cdots +z_n^2 = 1 - \Gamma(t)^{k+1}\}.
$$ 
In general, for $c \in \C$ nonzero, consider the level set of the quadratic polynomial
$$
Q_c:=\{z \in \C^{n+1}\;|\; z_1^2 + \cdots +z_n^2 = c\}.
$$
The value $\sqrt{c}$ is not uniquely determined, but if we fix the angle range to be $[0, 2 \pi)$, then we have only two of them. From now on, by $\sqrt{c}$ we always mean the one with the smaller angle. For example, we put $\sqrt{-1} = i$, while $-i$ also satisfies $(-i)^2 = -1$. We now define a map $\Phi_c : Q_c \rightarrow T^*S^{n-1}$ as follows: We first rotate the standard coordinate system by the angle of $\sqrt{c}$ on each coordinate, and then we send the rotated coordinate, say $u+iv$, to $(|u|^{-1}u, |u|v)$. The resulting map $\Phi_c: Q_c \rightarrow T^*S^{n-1}$ is then an exact symplectomorphism.

Under the identification $\pi^{-1} (\Gamma_j(t)) = T^{*}S^{n-1}$ via $\Phi_{1-\Gamma_j(t)^{k+1}}$, we define a subset $L_j \subset V_k$ by
$$
L_j : = \bigsqcup_{t \in [1, \infty)} \{\text{the zero section of $\pi^{-1} (\Gamma_j(t))$}\},
$$
for each $0 \leq j \leq k$. Then it is well-known that $L_j$ is a Lagrangian in $V_k$, which is called a \emph{Lefschetz thimble} associated to the critical value $\xi_j$; see \cite{KhSe02}, \cite{Se08}. 

\subsection{Real Lagrangians}\label{sec: defrealLag} We describe the Lefschetz thimble $L_j$ as the fixed point set of an anti-symplectic involution on $V_k$. For each $0\leq j \leq k$, consider the straight line $\ell_j$ in $\C$ which connects the origin and $\xi_j$. We denote the reflection on $\C$ whose axis is $\ell_j$ by $R_j: \C \rightarrow \C$. We define an involution on $\C^{n+1}$ by
$$
\rho_j :\C^{n+1} \rightarrow \C^{n+1}, \quad z \mapsto (R_j(z_0), -\overline z_0, \dots, -\overline z_n). 
$$
Then $\rho_j$ is an exact anti-symplectic involution on $\C^{n+1}$ with respect to the standard Liouville form. Observe that $\rho_j$ restricts to an anti-symplectic involution on the Milnor fiber $V_k$ with respect to the Liouville 1-form inherited from $\C^{n+1}$. We denote the restricted involution by the same notation $\rho_j : V_k \rightarrow V_k$. 

Since $\rho_j$ is an anti-symplectic involution, the fixed point set $\Fix(\rho_j)$ is a Lagrangian in $V_k$, which is called a \emph{real Lagrangian}. We note that if $k$ is even, then $\Fix(\rho_j)$ is connected, and if $k$ is odd, then $\Fix(\rho_j)$ consists of two connected components.

Now we again consider the Lefschetz fibration $\pi: V_k \rightarrow \R$ in \eqref{eq: leffibak}. Observe that the image $\pi(\Fix(\rho_j))$ lies on the line $\ell_j \subset \C$. More precisely, each connected component of the real Lagrangian $\Fix(\rho_j)$ is mapped either to the half line $\xi_j \cdot [1, \infty) \in \C$ or to the half line $-\xi_j \cdot [1, \infty] \in \C$. (The latter happens only when $k$ is odd.) Therefore, we can label the connected components of the real Lagrangians in the counter-clockwise order, say $\tilde L_0, \dots, \tilde L_k$.

\begin{example}
Consider the case when $k=1$. The anti-symplectic involutions are given by
$$
\rho(z) : = \rho_0 (z) = \rho_1(z) = (\overline z_0, -\overline z_0, \dots, -\overline z_n).
$$
Its fixed point set is then
$$
\Fix(\rho) = \{z \in \C^{n+1} \;|\; x_0^2 - y_1^2 - \cdots - y_n^2 = 1\}
$$
where $z_j = x_j + i y_j$. This has two connected components, namely
$\{x_0 > 0\}$ and $\{x_0 < 0\}$. The images of them under the Lefschetz fibration $\pi : V_1 \rightarrow \C$ correspond to $\{x \geq 1\}$ and $\{x \leq -1\}$ respectively, where $\xi_0 =1$ and $\xi_1 = -1$ are the critical values of $\pi$. We label the connected components as
$$
\tilde L_0 = \{x_0 > 0\}, \quad \tilde L_1 = \{x_0 < 0\}.
$$
\end{example}

\begin{proposition}
The connected component $\tilde L_j$ coincides with the Lefschetz thimble $L_j$.
\end{proposition}

\begin{proof}
This is just a direct computation. We work out the case $j=0$. Recall that the involution which defines $\tilde L_0$ is given by
$$
\rho_0(z) = (\overline z_0, -\overline z_1, \dots, -\overline z_n).
$$
Its fixed point set has possibly two connected components, but in any case the connected component $\tilde L_0$ is written by
\begin{align*}
\tilde L_0 &= \{z \in \C^{n+1}\;|\; x_0^{k+1} - y_1^2 - \cdots - y_n^2 = 1, \; x_0 \geq 1\}\\
&=\{z \in \C^{n+1}\;|\; y_1^2 + \cdots + y_n^2 = x_0^{k+1}-1, \; x_0 \geq 1 \}.
\end{align*}

On the other hand, the Lefschetz thimble $L_0$ is the union of the zero sections of $\pi^{-1}(\Gamma_0(t))$, $1\leq t < \infty$, under the identification $\pi^{-1}(\Gamma_0(t)) = T^*S^{n-1}$ (except for $t = 1$). This identification is precisely given by the map $\Phi_{1-\Gamma_0(t)^{k+1}}$ defined in Section \ref{sec: Lefthim}. Since $\Gamma_0(t) = t$ and hence $1-\Gamma_0(t)^{k+1} = 1-t^{k+1}$ is a positive real number in this case, the map $\Phi_{1-\Gamma_0(t)^{k+1}}$ is given by the formula
$$
z \in \pi^{-1}(\Gamma_0(t)) \longmapsto (-|y|^{-1}y, |y|x) \in T^*S^{n-1}.
$$
In particular the zero section of $T^*S^{n-1}$ corresponds to the set
$$
\pi^{-1}(\Gamma_0(t)) \cap \{x_1 =\cdots =x_n = 0\} = \{z \in \C^{n+1}\;|\; y_1^2 + \cdots + y_n^2 = \Gamma_0(t)^{k+1}-1\}.
$$
It follows that
\begin{align*}
L_0 &= \bigsqcup_{t \in [1, \infty)} \{\text{zero section of $\pi^{-1} (\Gamma_0(t))$}\} \\
&= \bigsqcup_{x_0 \in [1, \infty)} \{z \in \C^{n+1}\;|\; y_1^2 + \cdots + y_n^2 = x_0^{k+1}-1\}\\
&= \{z \in \C^{n+1}\;|\; y_1^2 + \cdots + y_n^2 = x_0^{k+1}-1, \; x_0 \geq 1 \} = \tilde L_0.
\end{align*}
This completes the proof.
\end{proof}

In short, the Lefschetz thimble $L_j$ is the connected component $\tilde L_j$ of the real Lagrangians. From now on, we will use the notation $L_j$, $0 \leq j \leq k$, for the component of the real Lagrangians.

\subsection{Real Lagrangians in plumbings} \label{sec: realplumbings}The Milnor fiber $V_k$ is Liouville isotopic to the linear plumbing $\#^{k} T^*S^n$ of cotangent bundles of the sphere, called the \emph{$A_k$-type plumbing}. In this section, we give a description of this fact, and we observe that the real Lagrangian $L_j$, $0\leq j \leq k$, can be seen as a cotangent fiber or a diagonal Lagrangian in the $A_k$-type plumbing $\#^{k} T^*S^n$.

\subsubsection{$A_k$-type plumbing} To fix notations, we briefly describe the plumbing of two cotangent bundles. The definition of the plumbing associated to the $A_k$-type Dynkin diagram is then a straightforward generalization.

Let $Q_0$ and $Q_1$ be closed manifolds of dimension $n$. Following \cite[Section 2]{AbSm12}, we define a \emph{local model of the  plumbing region} as follows.
\begin{equation}\label{eq: modelplumb}
R : = \{(x, y) \in \R^n \times \R^n \;|\; |x||y| \leq 1\}.
\end{equation}
Consider the ball $B^n \subset \R^n$ of radius $1/2$. Its disk cotangent bundle $DT^*B^n$ is identified with $B^n \times B^n$ via $ \sum {y_i dx_i} = (x_1, \dots, x_n, y_1, \dots, y_n)$, where $x=(x_1, ... ,x_n)$ is the base coordinate and $y=(y_1, ... ,y_n)$ is the fiber coordinate.
Furthermore, the disk cotangent bundle $DT^*B^n$ has the symplectic form $\sum_{i}dx_i \w dy_i$. We glue two copies of $DT^*B^n$ via the symplectomorphism $(x, y) \mapsto (y, -x)$. This is exactly how the local model $R$ looks like near the origin.

For each $j=0,1$, we take a metric $g_j$ on $Q_j$ locally flat near a fixed point $q_j$ and consider an isometric embedding
$$
(B_j^n,0) \hookrightarrow (Q_j,q_j)
$$
of the ball of radius $1/2$ centered at the origin. This induces a symplectic embedding 
$$
DT^*B_j^n \hookrightarrow DT^*Q_j
$$
for $j = 0, 1$.
We glue the disk cotangent bundles $DT^*Q_0$ and $DT^*Q_1$ using the map $(x, y) \mapsto (y, -x)$ on $DT^* B^n$ described above. Smoothing along the corner of the resulting space, we get a Liouville domain $DT^*Q_0 \# DT^*Q_1$. Its completion is called the \emph{plumbing} of $T^*Q_0$ and $T^*Q_1$ and is denoted by $T^*Q_0 \# T^*Q_1$.


In the plumbing $T^*Q_0 \# T^*Q_1$, there are two notable Lagrangians, namely \emph{cotangent fibers} (away from the plumbing region) and \emph{diagonal Lagrangians}. By a cotangent fiber in the plumbing, we mean the completion of a cotangent fiber in $DT^*Q_j$ (for some $j=0, 1$) away from the plumbing region. In the plumbing region, we have Lagrangians which can be written in the local model as
\begin{equation}\label{eq: diaglag}
D : = \{(x, y) \in R \;|\; \text{either $x_j= y_j$ or $x_j = -y_j$ for each $1\leq j \leq n$}\}.
\end{equation}
We call Lagrangians obtained by completing $D$ \emph{diagonal Lagrangians}.

\subsubsection{The Milnor fiber $V_k$ as a plumbing} To describe the Milnor fiber $V_k$ as a plumbing, we use the following neighborhood theorem. We provide its proof for later use in Section \ref{sec: fibersdiagonals} although it is fairly well-known.

\begin{proposition}\label{prop: plumbnbdthm}
Let $S_j$, $1 \leq j \leq k$, be closed exact Lagrangians in a symplectic manifold $(W,\omega)$. Suppose $S_j$'s are in the $A_k$-configuration, i.e. $S_{j}$ and $S_{j+1}$  with $1 \leq j \leq k-1$ intersect transversely at one point and there exist no other intersections between these Lagrangians. Then there exists a symplectic embedding $\Phi : \nu(\bigcup_{j} S_j) \rightarrow W$ from a neighborhood $\nu(\bigcup_{j} S_j)$ in the $A_k$-type plumbing $ \#_{j} DT^*S_j$ into $W$, which sends $S_j$ in the plumbing $\#_{j} DT^*S_j$ to $S_j$ in $W$.
\end{proposition}

For a proof, we use the following lemma from \cite[Proposition 3.4.1]{Po94}.

\begin{lemma}[Po\'zniak]\label{lemmaofPozniak}
Let $(M,\omega)$ be a symplectic manifold of dimension $2n$ and let $L_0,L_1$ be two Lagrangian submanifolds which intersect transversely at $q \in L_0 \cap L_1$. Then there exist a neighborhood $V$ of $q$ in $\R^n \times \R^n$, a neighborhood $U$ of $q$ in $M$ and a symplectomorphism $\psi : (U,\omega) \to (V,\omega_{std})$ such that
$$ \psi(L_0 \cap U) =  \Lambda_0 \cap V \text{ and } \psi(L_1 \cap U) = \Lambda_1 \cap V.$$ 
where $\Lambda_0 = \R^n \times \{0\}$ and $\Lambda_1 = \{0\} \times \R^n$. 
\end{lemma}

\begin{proof}[Proof of Proposition \ref{prop: plumbnbdthm}]
We provide an argument for the case when $k=2$; the general case follows from the same argument. 

Suppose $S_0$ and $S_1$ intersect at $q \in W$ transversely. Applying Lemma \ref{lemmaofPozniak} to the Lagrangians $S_0$ and $S_1$ we have an open neighborhood $U$ of $q$ in $W$ as in the lemma. Shrinking $U$ and scaling $\omega$ if necessary, we may identify $U_0:=U \cap S_0$ with $B^n \times \{0\} \subset \R^n \times \{ 0\}$ and $U_1 := U \cap S_1$ with $\{0\} \times B^n \subset \{0\}\times \R^n$ where $B^n$ is the open ball of radius $1/2$. Accordingly, we have an identification $DT^*U_0 = B^n \times B^n =DT^* U_1 \subset \R^n \times \R^n$, which is exactly the same as what happened in the plumbing region.


We take a metric $g$ on $W$ such that $g = \sum dx_i \otimes dx_i + dy_i \otimes dy_i$ on $U$.   The metric $g$ induces a bundle isomorphism $\psi_j : T^*S_j \to T S_j^{\perp_g}$, $ v^* \mapsto u$ where $u \in T_p S^{\perp_g}$ is chosen so that $\omega(\cdot, u) = v^* \in T_p^* S_j$.
We define a smooth embedding  $\Psi_j : DT^* S_j \to W$ by
$$ \Psi_j(q_j,v^*_j) = \exp_{q_j} ( \psi(v_j^*) ) \text{ for } (q_j,v_j^*) \in DT^*S_j.$$
We may shrink the radius of each of the disk cotangent bundles to ensure that both maps $\Psi_0$ and $\Psi_1$ are embeddings.
 
 Under the identification $DT^*U_0 =B^n \times B^n$, the restriction of $\Psi_0$ to $DT^*U_0 \subset DT^* S_0$ is identified with the inclusion map $B^n \times B^n \hookrightarrow \R^n \times \R^n$ in the neighborhood obtained in Lemma \ref{lemmaofPozniak}. Similarly, the restriction of $\Psi_1$ to $DT^*U_1$ can be identified with the inclusion map $B^n \times B^n \hookrightarrow \R^n \times \R^n$. This shows that $\Psi_0$ and $\Psi_1$ glue together and yield a map $\Psi : DT^*S_0 \# DT^*S_1 \to W$.

Furthermore we have $\Psi^* \omega = \omega_{can}$ on each $S_j \subset DT^* S_j$. Since  $[\Psi^* \omega]  = [\omega_{can}] \in H^2 (DT^* S_0 \# DT^*S_1)$, we can apply the Moser's trick to find a symplectic embedding $\Phi : \nu(\bigcup_{j} S_j) \to W$ of a neighborhood $\nu(S_0 \cup S_1) \subset DT^* S_0 \# DT^* S_1$ of $S_0 \cup S_1$ into $W$. Since $\Psi^* \omega |_{S_j} = \omega_{can}|_{S_j}$ and $\Psi$ sends $S_j \subset DT*S_0 \cup DT*S_1$ to $S_j \subset W$ for each $j=0,1$, so does the map $\Phi$.
\end{proof}

We can now identify the Milnor fiber $V_k$ with the plumbing of $T^*S^n$'s associated to the $A_k$-type Dynkin diagram. Consider the Lefschetz fibration 
$$
\pi: V_k \rightarrow \C, \quad z \mapsto z_0.
$$ 
As in Section \ref{sec: Lefthim}, denote the critical values by $\xi_0 = 1, \xi_1, \dots, \xi_k = e^{\frac{2k\pi i}{k+1}}$; the $(k+1)$-th roots of unity. We define a subset $S_j \subset V_k$ for each $0 \leq j \leq k$ as follows. Consider the line segment in the base $\C$ which connects $\xi_j$ and $\xi_{j+1}$ (here we conventionally put $\xi_{k+1} = \xi_0$). Observe that at each point  on the line segment except for the end points, the fiber is a copy of $T^*S^{n-1}$, and at the end points the fiber is a copy of $T^*S^{n-1}$ with the zero section collapsed to a point. Define $S_j$ to be the union of the zero section on each fibers over the segment. In terms of Picard-Lefschetz theory, $S_j$ is called a \emph{matching cycle} in the total space of the Lefschetz fibration, and it is an exact Lagrangian sphere, see \cite[Section 16]{Se08} 

We, in particular, consider the Lagrangian spheres $S_0, \dots, S_{k-1}$ in $V_k$. Note that they are in the $A_k$ configuration by construction. Think of the $A_k$-type plumbing of cotangent bundles of $S_j$'s. We can embed a neighborhood of the union of $S_j$'s in the plumbing into a neighborhood of the union of $S_j$'s in $V_k$ by Proposition \ref{prop: plumbnbdthm}. Since the completion of the neighborhood of $S_j$'s in $V_k$ contains all the critical points of $\pi$, we conclude that $V_k$ is exact symplectomorphic to the $A_k$-type plumbing $\#^{k}_{j}T^*S_j =\#^kT^*S^n$.

\begin{remark}
The identification of $V_k$ with the $A_k$-type plumbing $\#^kT^*S^n$ depends on the choice of matching cycles of the Lefschetz fibration $\pi: V_k \rightarrow \C$. Our discussion below is with respect to the above choice of matching cycles $S_0, \dots, S_{k-1}$.
\end{remark}

\subsubsection{Cotangent fibers and diagonal Lagrangians}\label{sec: fibersdiagonals} In the Milnor fiber $V_k$, we have defined real Lagrangians, $L_j$ with $0 \leq j \leq k$, which can be seen as Lefschetz thimbles. In this section, we observe the following.

\begin{proposition}\label{prop: ident_diag_fiber}
Under the identification $V_k = \#^k_j T^*S_j = \#^k T^*S^n$, the two real Lagrangians $L_0$ and $L_k$ are cotangent fibers of $T^*S_0$ and $T^*S_{k-1}$, respectively, away from the plumbing regions, and the other real Lagrangians are diagonal Lagrangians in the plumbing regions. 
\end{proposition}

\begin{proof}
This follows from the intersection configurations between $L_j$'s and $S_j$'s. For a simper presentation, let us assume that $k=2$ (the other cases are then straightforward). We have two Lagrangian spheres $S_0$ and $S_1$ in the $A_2$-configuration and we have three Lagrangians $L_0$, $L_1$, and $L_2$ located as in Figure \ref{real_plumb}. 

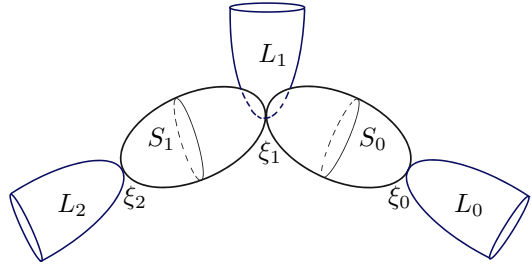
\begin{figure}[htb]
	\centering
	\begin{overpic}[width=200pt,clip]{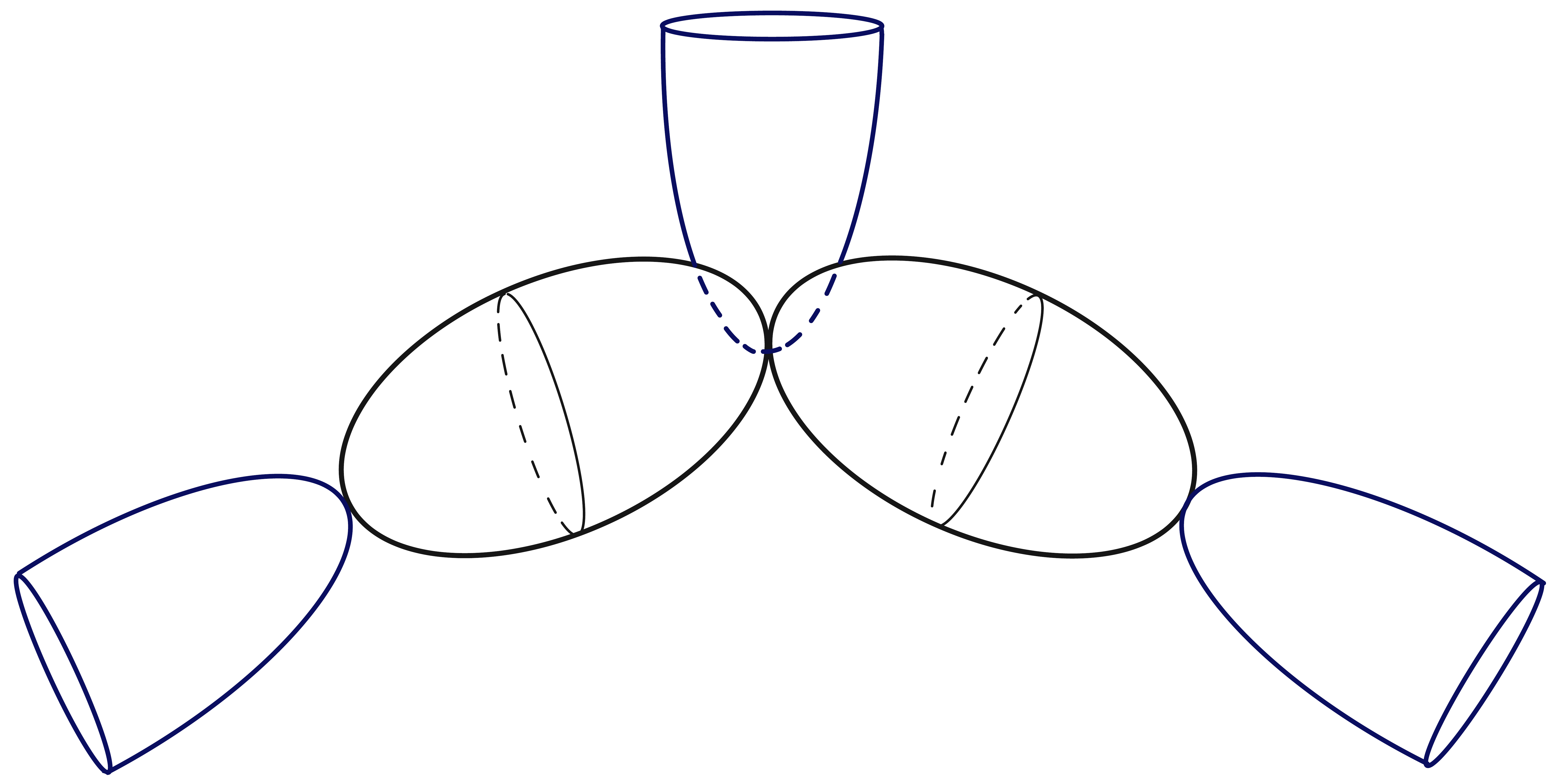}
	\put(170,20){$L_0$}
	\put(96,77){$L_1$}
	\put(20,20){$L_2$}
	\put(134,45){$S_0$}
	\put(54,45){$S_1$}
	\put(145,23){$\xi_0$}
	\put(45,23){$\xi_2$}
	\put(96,40){$\xi_1$}
	
	\end{overpic}
	\caption{Real Lagrangians with matching cycles}
	\label{real_plumb}
\end{figure}

In particular $L_0$ intersects $S_0$ transversely at the critical point $\xi_0 = 1$. In view of the proof of Proposition \ref{prop: plumbnbdthm}, we can readily modify the embedding $\Phi$ in the statement of Proposition \ref{prop: plumbnbdthm} to have the property that $\Phi$ sends a local neighborhood of $\xi_0$ in $L_0$ to a fiber in the plumbing $T^*S_0 \# T^*S_1$ which comes from a fiber of the $T^*S_0$-factor. The same argument applies to $L_2$ since $L_2$ intersects $S_1$ transversely at one point.

Now let us consider the other Lagrangian $L_1$. This Lagrangian intersects $S_0$ and $S_1$ transversely at a common point, namely at $\xi_1$. In view of Proposition \ref{prop: plumbnbdthm}, we can take a local chart $U$ of $\xi_1$ in $V_2$, identified symplectically with (a neighborhood of) the model plumbing $R$ defined in \eqref{eq: modelplumb} such that $S_0$ corresponds to the part $\{y = 0\}$ (the zero section) and $S_1$ corresponds to the part $\{x = 0\}$ (a cotangent fiber).

In this chart, since $L_1$ is transverse to $S_1$, it can be written as the graph of an exact 1-form, say $df$, where $f$ is a smooth function $f: \R^n \rightarrow \R$ of the $x$-coordinates. Since $L_1$ intersects $S_0$ at the origin $(x,y)=(0, 0)$ (in the chart), it follows that $df=0$ at the origin. In other words, the origin in $R$ is the unique critical point of $f$. Furthermore, since the intersection is transverse, it follows that the origin is a non-degenerate critical point of $f$.

Note, from the equation \eqref{eq: diaglag}, that the Lagrangian $L_1$ is diagonal in $R$ if and only if the 1-form $df$ is of the following form
$$
df = \sum_{i = 1}^n \lda_ix_i dx_i
$$
where $\lda_i = \pm 1$. That is, the coefficient functions of $df$ are all linear with slope $1$ or $-1$. This is precisely the case when we take the $x$-coordinates for $S_0$ so that $f$ is written near the origin as a complex quadratic polynomial. This is always possible thanks to the Morse lemma. It follows that $L_1$ can be written as a diagonal Lagrangian in a plumbing region. This completes the proof.
\end{proof}

\subsection{Computing the group structure of (v-shaped) wrapped Floer homology}\label{sec: computegroups} An advantage of regarding the Lefschetz thimbles as real Lagrangians is that this enables us to compute wrapped Floer homology explicitly. In this section, we outline the computation of the group structure of (v-shaped) wrapped Floer homology of $L_j$'s.

\subsubsection{A Morse--Bott spectral sequence in v-shaped wrapped Floer homology} The group structure of the wrapped Floer homology of $L_0$ is computed in \cite[Section 7.4.2]{KiKwLe18} using a Morse--Bott spectral sequence. Actually, we can compute $\HW_*(L_j)$ for each $j$ using the same computational technique. One then finds that the group $\HW_*(L_j)$ are isomorphic to each other for all $0\leq j\leq k$. This is because their Morse--Bott setups (that is, Morse--Bott submanifolds and their Maslov indices) are the same, so they have the same $E^1$-page of the spectral sequence in \cite[Theorem 2.1]{KiKwLe18}. We consequently have the following computation.

\begin{proposition}\label{prop: groupcompu}
Let $n \geq 3$. For every $0 \leq j \leq k$, the wrapped Floer homology group of $L_j$ is given by
$$
\HW_*(L_j) = \begin{cases} \Z_2 & *=0, \{(n-2)(k+1) +2\}N -n +1, \{(n-2)(k+1) +2\}N \text{ with } N\in \N; \\ 0 & \text{otherwise}.  \end{cases}
$$
\end{proposition}

\begin{remark}
For $n \geq 3$, the topological conditions \eqref{eq: topconditions} are all satisfied with $(V_k, L_j)$.
\end{remark}

\subsubsection{V-shaped wrapped Floer homology of $L_j$} \label{sec: compgroupvshaped} As the ordinary case, we can use a Morse--Bott spectral sequence to compute the group structure of the v-shaped wrapped Floer homology $\widecheck{\HW}_*(L_j)$.

To state a spectral sequence we fix notations as follows. Let $(W, d\lda)$ be a Liouville domain with an admissible Lagrangian $L$. Denote its Legendrian boundary in $\Sigma: = \p W$ by $\mathcal{L}$. For technical simplicity, we assume in addition to the condition \eqref{eq: topconditions} that $\pi_1(\mathcal{L})=0$. 

Assume that every Reeb chord in $(\Sigma, \alpha, \mathcal{L})$ is of Morse--Bott type in the sense of \cite[Definition 2.6]{KiKwLe18}. Define the \emph{v-spectrum} of $(\Sigma, \alpha, \mathcal{L})$ by
$$
\widecheck{\Spec}(\Sigma, \alpha, \mathcal{L}) : = \{T \in \R \;|\; \phi_R^T(x) \in \mathcal{L} \text{ for some $x \in \mathcal{L}$}\},
$$
where $\phi_R^t(x)$ denotes the time-$t$ Reeb flow of the contact form $\alpha$. For each $T \in \widecheck{\Spec}(\Sigma, \alpha,\mathcal{L})$, we define the corresponding Morse--Bott submanifold by
$$
\mathcal{L}_T : =\{x \in \mathcal{L} \;|\; \phi_R^T(x) \in \mathcal{L}\}. 
$$ 
We arrange the v-spectrum as 
$$
\widecheck{\Spec}(\Sigma, \alpha,\mathcal{L}) = \{\dots, T_{-2}< T_{-1}< T_0 = 0 < T_1< T_2, \dots \}.
$$
\begin{remark}
If one has arranged the (non v-shaped) spectrum as
$$
\Spec(\Sigma, \alpha, \mathcal{L}) = \{T_0 = 0 < T_1< T_2 < \cdots\},
$$
then the v-shaped spectrum is arranged just by putting $T_{-p} = - T_p$ for $p \in \N$.
\end{remark}
\begin{theorem} \label{vMBss}
There exists a spectral sequence $\{(E^r, d^r)\}_{r \in \Z_{\geq 0}}$ converging to the v-shaped wrapped Floer homology $\widecheck{\HW}_*(L;W)$ such that  
$$
E^1_{pq} =  \begin{cases} \displaystyle  H_{p+q - \sh(\mathcal{L}_{T_p})+\frac{1}{2}\dim L}(\mathcal{L}_{T_p} ; \Z_2) & p\neq 0, \\ \displaystyle H_{q + \dim \mathcal{L}} (\mathcal{L}; \Z_2) & p = 0.  \end{cases}
$$
where $\sh(\mathcal{L}_{T_p}) = \begin{cases} \mu(\mathcal{L}_{T_p}) - \frac{1}{2} (\dim \mathcal{L}_{T_p} -1) & p >0 \\ \mu(\mathcal{L}_{T_p}) - \frac{1}{2} (\dim \mathcal{L}_{T_p} +1) & p < 0 \end{cases}$.
\end{theorem}

\begin{remark}\label{rem: negativeCZind} \
\begin{enumerate}
\item Our grading convention is such that $\widecheck{\HW}^{(-\epsilon, \epsilon)}_*(L) \cong H_{* + n-1}(\mathcal{L})$ for sufficiently small $\epsilon >0$.
\item Here $\mu(\mathcal{L}_{T_p})$ denotes the Maslov index (also referred as Robbin--Salamon index) of a Morse--Bott component $\mathcal{L}_{T_p}$, see \cite[Section 2.2.2]{KiKwLe18} for a definition. By the inverse property of the index, we have $\mu(\mathcal{L}_{T_{-p}}) = - \mu(\mathcal{L}_{T_p})$ for each non-zero $p \in \Z$.
\item The construction of the Morse--Bott spectral sequence is fairly standard and is essentially the same as the ordinary wrapped Floer homology case. We use an action filtration with respect to the action functional \eqref{eq: Hamactfunl} which is adapted to the Morse--Bott setup.
\end{enumerate}
\end{remark}

\subsubsection{Computing the group $\widecheck{\HW}_*(L_j)$} \label{ex: groupcomputation}
Applying the spectral sequence to $L_j$ in $V_k$, we can compute the group structure of its v-shaped wrapped Floer homology. First of all, the computation of the $E^1$-page can be done exactly as in wrapped Floer homology case \cite[Section 7.4.2]{KiKwLe18}. The only computational differences are that we have the homology of the Legendrian $\mathcal{L}_j$ in the column $p=0$ instead of the relative homology of the pair $(L_j, \mathcal{L}_j)$, and we have the homology of the Morse--Bott submanifold $\mathcal{L}_{T_{-p}}$ in the corresponding column for each $p \in \N$. In particular the latter one is the same as the homology of $\mathcal{L}_{T_p}$ but with a different degree shift. As noted in Remark \ref{rem: negativeCZind}, the degree shift for $\mathcal{L}_{T_{-p}}$ differs only by sign from that of $\mathcal{L}_{T_{p}}$.

We present the computation of $\widecheck{\HW}_*(L_0)$ in more detail. Firstly, to fit our situation into a Morse--Bott setup, we deform the (non-completed) Milnor fiber $V_k \cap B^{2n+2}$ into the following Liouville domain
$$
W = W_k := \{z \in \C^{n+1}\;|\; z_0^{k+1} + z_1^2 + \cdots + z_n^2 = \zeta(|z|)\} \cap B^{2n+2}
$$   
where $B^{2n+2} \subset \C^{n+1}$ is the unit ball and $\zeta : \R \rightarrow \R$ is a monotone smooth function such that $\zeta(r)=1$ for $0 \leq r \leq 1/4$ and $\zeta(r) = 0$ for $3/4 \leq r \leq 1$. The boundary $\p W =: \Sigma$ is then an $A_k$-type Brieskorn manifold, and the Legendrian $\mathcal{L} = \mathcal{L}_0 \subset \Sigma$ in the boundary is now written as 
$$
\mathcal{L} = \{z \in \C^{n+1} \;|\; z_0^{k+1} + z_1^2 + \cdots + z_n^2 = 0, \; |z|^2 =1, \; y_0=x_1=\cdots =x_n=0\},
$$
where $z_j = x_j + i y_j$. 

The Brieskorn manifold $\Sigma$ carries a Morse--Bott type contact form $\alpha$ given by
\begin{equation}\label{eq: MBcontactform}
\alpha : = \frac{i}{2}\left\{(k+1)(z_0d\overline z_0 - \overline z_0 dz_0) + 2 \sum_{j=1}^n (z_jd\overline z_j - \overline z_j dz_j) \right\}
\end{equation}
restricted to $\Sigma$. The corresponding Reeb flow is explicitly given as
\begin{equation}\label{eq: periodicReebflow}
\phi_{R}^t(z) = (e^{it/{k+1}}z_0, e^{it/2}z_1, \dots, e^{it/2}z_{n}),
\end{equation}
and the corresponding v-spectrum is given by
\begin{equation}\label{eq: vspectrum}
\widecheck{\Spec}(\Sigma, \alpha, \mathcal{L}) = \{N \cdot 2(k+1)\pi \;|\; N \in \Z \}.
\end{equation}

It is straightforward to see that, for each $T \in \widecheck{\Spec}(\Sigma, \alpha, \mathcal{L})$, the Morse--Bott component $\mathcal{L}_T$ is identical to $\mathcal{L}$, which is topologically equivalent to $S^{n-1}$. Furthermore, by \cite[Lemma 7.2]{KiKwLe18}, we obtain the indices $\mu(\mathcal{L}_T)$ as follows.
$$
\mu(\mathcal{L}_{N \cdot 2(k+1)\pi}) = - \mu({\mathcal{L}_{-N \cdot 2(k+1)\pi}}) = \{2+(n-2)(k+1)\} N. 
$$

Using these data, we determine the $E^1$-page of the spectral sequence. It turns out that if $n \geq 3$, the spectral sequence terminates at the $E^1$-page for degree reasons. For example if $k=2$ or $3$ and $n=3$ the $E^1$-page is given as Figure \ref{fig: E1example}. Furthermore, for a fixed $k$ and $n$, the $E^1$-page for $\widecheck{\HW}_*(L_j)$ does not depend on $0\leq j \leq k$. We therefore obtain the following computation.

\begin{proposition}\label{prop: groupvshapedcompu}
For $n \geq 3$ and $0 \leq j \leq k$,
$$
\widecheck{\HW}_*(L_j) = \begin{cases} \Z_2 & * = \{(n-2)(k+1) +2\}N -n +1, \{(n-2)(k+1) +2\}N \text{ with $N \in \Z$}; \\ 0 & \text{otherwise}.  \end{cases}
$$
\end{proposition}

\begin{figure}[htp]
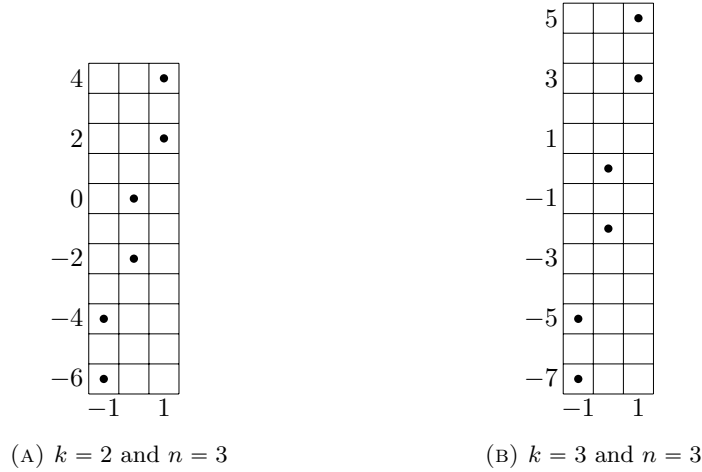

\begin{subfigure}[b]{0.25\textwidth}
\centering
\begin{sseq}{{-1}...1}
{{-6}...4}
\ssmoveto 0 {-2}
\ssdropbull
\ssmove 0 2
\ssdropbull

\ssmoveto 1 2
\ssdropbull
\ssmove 0 2
\ssdropbull


\ssmoveto {-1} {-6}
\ssdropbull
\ssmove 0 2
\ssdropbull


\end{sseq}
\caption{$k=2$ and $n=3$}
\end{subfigure}\hspace{7.0em}
\begin{subfigure}[b]{0.25\textwidth}
\centering
\begin{sseq}{{-1}...1}
{{-7}...5}
\ssmoveto 0 {-2}
\ssdropbull
\ssmove 0 2
\ssdropbull

\ssmoveto 1 3
\ssdropbull
\ssmove 0 2
\ssdropbull


\ssmoveto {-1} {-7}
\ssdropbull
\ssmove 0 2
\ssdropbull


\end{sseq}

\caption{$k=3$ and $n=3$}
\end{subfigure}
\caption{The $E^1$-page in examples}
\label{fig: E1example}
\end{figure}

\subsection{Index-positivity of the real Legendrians in $A_k$-type Brieskorn manifolds} \label{sec: indposak} In this section we show that $A_k$-type Brieskorn manifolds with the real Legendrians are product-index-positive for $n \geq 3$.
For the sake of simplicity, we only deal with the case of the first real Legendrian $\mathcal{L} := \mathcal{L}_0 = \p L_0$ assuming $k$ is even; the assertion for the other cases follows from the corresponding computation with minor numerical changes.

Recall that the $A_k$-type Brieskorn manifold is given by
$$
\Sigma =\{z \in \C^{n+1} \;|\; z_0^{k+1}+z_1^2  + \cdots + z_n^{2} = 0,\; |z| = 1\}.
$$
The real Legendrian $\mathcal{L}$ is given by
$$
\mathcal{L} = \{z \in \C^{n+1} \;|\; z \in \Sigma, \; y_0 = x_1 = \cdots = x_n = 0\}
$$
where $z = x + iy$. We introduced in Section \ref{ex: groupcomputation} a Morse--Bott contact form $\alpha$ \eqref{eq: MBcontactform} on $\Sigma$ whose Reeb flow \eqref{eq: periodicReebflow} is periodic. Index-positivity for those periodic Reeb orbits in $A_k$-type Brieskorn manifolds is already discussed in several precedent works, for example, see \cite[Section 5.4]{KwvK16} and \cite[Lemma 1]{Ue16}. Here we check the Reeb chords case with Remark \ref{rem: indposMB}.


Recall that the periods of the Reeb chords in $(\Sigma, \alpha, \mathcal{L})$ are given by $2(k+1)\pi \cdot N$ for $N \in \N$ and the corresponding Morse--Bott submanifold of Reeb chords is given identically by $\mathcal{L}_{2(k+1)\pi \cdot N} = \mathcal{L}$. 
By \cite[Lemma 7.2]{KiKwLe18}, we obtain the index as follows.
$$
\mu(\mathcal{L}_{N \cdot 2(k+1)\pi}) = \{2+(n-2)(k+1)\} N. 
$$
Let $c$ be a Reeb chord in a Morse--Bott submanifold $\mathcal{L}_{N \cdot 2(k+1)\pi}$. Since $k\geq 1$ by the definition, we find that
$$
\mu(\mathcal{L}_{N \cdot 2(k+1)\pi}) =  \{2+(n-2)(k+1)\} N \geq 2n -2. 
$$
In view of Remark \ref{rem: indposMB}, since $\dim \mathcal{L}_{N \cdot 2(k+1)\pi} = \dim \mathcal{L} = n-1$, it now suffices to have that
$$
2n -2> \frac{3}{2} + \frac{1}{2}\dim \mathcal{L}_{N \cdot 2(k+1)\pi} = \frac{n+2}{2}.
$$
This holds if $n \geq 3$. We therefore conclude the following.
\begin{proposition}\label{cor: indexpositicityinexample}
For $n\geq 3$, the triple $(\Sigma, \xi, \mathcal{L})$ is product-index-positive.
\end{proposition}

\section{Seidel operator on v-shaped wrapped Floer homology}\label{sec: seideloperator}

\subsection{Seidel operator} \label{sec: Seideloperator} In this section, we describe an open string analogue of the Seidel representation for Liouville domains and admissible Lagrangians. This associates to a path of Hamiltonian diffeomorphisms a graded group isomorphism on Floer homology, as in Theorem \ref{thm: Sgoperatoringeneral}. An important feature of the Seidel representation for our purpose is the \emph{module property} in Theorem \ref{thm: prodproperty}.  

An open string version of the Seidel representation was introduced in \cite{HuLa10} for \emph{closed} symplectic manifolds, and versions for symplectic homology were studied in \cite{Ri14}, \cite{Ue19}. Below we shall focus on differences from previous works.

\subsubsection{Setup}
Let $(W, d\lda)$ be a Liouville domain and $L$ an admissible Lagrangian in $W$ with the topological assumptions \eqref{eq: topconditions}. Let $g:[0, 1] \rightarrow \Ham(\widehat W)$ be a path of Hamiltonian diffeomorphisms on $\widehat W$. We require that
\begin{equation}\label{eq: assumptionong1}
\text{$p \in \widehat L$ if and only if $g_0(p) \in \widehat L$ and $g_1(p) \in \widehat L$.}
\end{equation}
Denote the space of paths in $\widehat W$ relative to $\widehat L$ by 
$$
\PP_{\widehat L}(\widehat W) : = \{x:[0, 1] \rightarrow \widehat W \;|\; x(0), x(1) \in \widehat L\}.
$$
Then $g$ acts on $\PP_{\widehat L}(\widehat W)$ by
$$
(g \cdot x) (t) : = g_t(x(t)).
$$

\subsubsection{Correspondences} Let $H: \widehat W \rightarrow \R$ be a v-shaped admissible Hamiltonian. Let $K^g: [0, 1] \times \widehat W \rightarrow \R$ be a Hamiltonian which generates the path $g_t$, that is,
$$
 d K^g (t,g_t(x)) ( \cdot) = \omega \left( \cdot, \frac{\p g_t(x)}{\p t} \bigg|_{t=0}  \right).
$$
We define a Hamiltonian $g_*H:[0,1]\times \widehat{W} \rightarrow \R$, which we call the \emph{pushforward of $H$ by $g$}, as follows.
$$
(g_*H)(t, p) :=H(t,g_t^{-1}(p))+ K^g(t,p). 
$$
Then the Hamiltonian flow of $g_*H$ is given by the composition $g_t \circ \phi_H^t$. In particular, due to the condition \eqref{eq: assumptionong1}, for a Hamiltonian $1$-chord $x$ of $H$ relative to $\widehat L$, the path $g\cdot x$ is a Hamiltonian 1-chord of $g_*H$ relative to $\widehat L$.

Let $J$ be an admissible almost complex structure on $\widehat W$. Define the \emph{pushforward $g_* J$ of $J$} to be
$$
g_* J: = dg_t \circ J \circ dg_t^{-1}.
$$
Let $u : \R \times [0,1] \rightarrow \widehat W$ be a Floer strip with respect to the pair $(H, J)$ from $x_- \in \mathcal{P}_{\widehat L}(H)$ to $x_+ \in \mathcal{P}_{\widehat L}(H)$. Then it is straightforward to see that the map $g \cdot u$ is a Floer strip from $g \cdot x_-$ to $g \cdot x_+$ with respect to the pair $(g_*H, g_*J)$. The following can be shown by the same argument as in \cite[Section 3.1]{HuLa10}.

\begin{lemma}\label{lem: sendfloerstrip}\

\begin{enumerate}
\item The map $x \mapsto g\cdot x$ gives a one-to-one correspondence 
$$
\PP_{\widehat L}(H) \leftrightarrow \PP_{\widehat L}(g_*H).
$$
\item The map $u \mapsto g\cdot u$ gives a one-to-one correspondence
$$
\mathcal{M}(x_-, x_+; H, J) \leftrightarrow \mathcal{M}(g \cdot x_-, g \cdot x_+; g_*H, g_*J).
$$
\item The pair $(H, J)$ is regular if and only if the pair $(g_*H, g_*J)$ is regular.
\end{enumerate}
\end{lemma} 

The above lemma allows us to define a Floer homology of the pair $(g_*H, g_*J)$. The chain complex $\CF(L;g_*H, g_*J)$ is generated by the elements of $\PP_{\widehat L}(g_*H)$ and the differential $\p$ counts the elements of the zero-dimensional component of the moduli space $\mathcal{M}(g \cdot x_-, g \cdot x_+; g_*H, g_*J)$. We get a well-defined homology $\HF(L;g_*H, g_*J)$ of the chain complex $(\CF(L;g_*H, g_*J), \p)$, which is group isomorphic to the Floer homology $\HF(L;H, J)$ via the correspondence $x \mapsto g\cdot x$. We denote the isomorphism by
$$
S_g : \HF(L;H, J) \rightarrow \HF(L;g_*H, g_*J),
$$
which we call a \emph{Seidel operator}.

\subsubsection{Degree shift by $S_g$} \label{sec: deg of I(g)} The isomorphism $S_g$ actually comes with a degree shift given by a Maslov type index $I(g)$ of the path $g$. The index $I(g)$ of $g$ is defined as follows. (See also \cite[Definition 3.4]{HuLa10}.) Note that for each v-shaped admissible Hamiltonian $H$, the set $\PP_{\widehat L}(H)$ is non-empty since there exist constant chords which generate $\widecheck{\HW}^{(-\epsilon, \epsilon)}(L) \cong H(\p L) \neq 0$ for sufficiently small $\epsilon >0$. Pick $x \in \PP_{\widehat L}(H)$. Under our topological assumptions \eqref{eq: topconditions}, the paths $x$ and $g\cdot x$ are contractible. As in Section \ref{sec: Floerchaincomplex}, we take a capping half disk for $x$, and this induces a trivialization of $T\widehat W$ along $x$, say
$$
\tau_x(t) : T_{x(t)}T\widehat W \rightarrow \R^{2n}.
$$
We also take a capping half disk for $g\cdot x$, so we obtain a trivialization $T\widehat W$ along $g\cdot x$, which we denote by
$$
\tau_{g\cdot x}(t) : T_{g\cdot x(t)}T\widehat W \rightarrow \R^{2n}.
$$
Define a path of symplectic matrices $\ell_x: [0, 1] \rightarrow Sp(2n)$ by
$$
\ell_x (t):=\tau_{g \cdot x}(t) \circ dg_t(g_0 (x(t))) \circ \tau_{g_0 \cdot x }(t)^{-1}.
$$
We define the \emph{index $I_x(g)$ of $g$} with respect to $x$ by
$$
I_x(g) : = \mu_{\RS} (\ell_x \Lambda_0, \Lambda_0)
$$
where $\Lambda_0 \subset \R^{2n}$ is the horizontal Lagrangian, and $\mu_{\RS}$ denotes the Robbin--Salamon index in \cite{RoSa93}. The indices are independent of  choices of capping half disks due to the topological conditions \eqref{eq: topconditions}. 

\begin{lemma}
The index $I_x(g)$ does not depend on the choice of $x$.
\end{lemma}

\begin{proof}
Let $x_0$ and $x_1$ be contractible paths in $\mathcal{P}_{\widehat L}(W)$. Then we have a homotopy $x_s$ with $s \in [0,1]$ in $\mathcal{P}_{\widehat L}(W)$ from $x_0$ to $x_1$. This yields a homotopy between $g \cdot x_0$ and $g \cdot x_1$ and hence a homotopy $
\ell_{x_s}$ of paths of symplectic matrices. Note that $\ell_{x_s}$ is not a loop, but the corresponding path of Lagrangians $\ell_{x_s}\Lambda_0$ in $\R^{2n}$ is a loop. This is because of the assumption \eqref{eq: assumptionong1} on $g$; indeed, we have that $Y \in T_{x_s(j)}L$ if and only if $dg_j(x_s(j))(Y) \in T_{g \cdot x(j)}L$ for $j=0,1$. Therefore $\ell_{x_s} \Lambda$ with $s \in [0,1]$ provides a homotopy of loops of Lagrangians in $\R^{2n}$ from $\ell_{x_0} \Lambda_0$ to $\ell_{x_1} \Lambda_0$. It follows from the homotopy invariance of Maslov index for loops of Lagrangians that 
$$
I_{x_0}(g) = \mu(\ell_{x_0} \Lambda_0, \Lambda_0) = \mu(\ell_{x_1} \Lambda_0, \Lambda_0) = I_{x_1}(g).
$$
This completes the proof.
\end{proof}

In view of the above lemma, we abbreviate the notation by $I(g) : = I_x(g)$. For the next proposition we introduce the following additional condition on $g$.
\begin{equation}\label{eq: assumeg2}
\text{The starting point $g_0$ commutes with the Hamiltonian flow $\phi_H^t$ i.e. $g_0^{-1} \circ \phi_H^t \circ g_0 = \phi_H^t$.}
\end{equation}
We remark that this condition is obviously true if $g_0 = \id$.

\begin{proposition}\label{prop: degreeshift}
Assume the condition \eqref{eq: assumeg2}. For a contractible Hamiltonian 1-chord $x$ relative to $\widehat L$, the Maslov index of $g\cdot x$ is given by
$$
\mu(g \cdot x) = \mu(x) + I(g).
$$
\end{proposition}

\begin{proof}
Let $\Phi_x:[0, 1] \rightarrow Sp(2n)$ be the path of symplectic matrices defined by
$$
\Phi_x(t) : = \tau_x(t) \circ d\phi_H^t(x(0)) \circ \tau_x(0)^{-1}.
$$
Then the Maslov index of $x$ is by definition given by
$$
\mu(x) = \mu_{\RS}(\Phi_x\Lambda_0, \Lambda_0).
$$  
Similarly, the Maslov index of $g \cdot x$ is defined by
$$
\mu(g \cdot x) = \mu_{\RS}(\Phi_{g \cdot x}\Lambda_0, \Lambda_0)
$$
where the path $\Phi_{g \cdot x}$ is given by
$$
\Phi_{g \cdot x}(t) : = \tau_{g\cdot x}(t) \circ d(g_t \circ \phi_H^t)(g \cdot x(0)) \tau_{g \cdot x}(0)^{-1}.
$$
We compute, using the assumption that $\phi_H^t \circ g_0 = g_0 \circ \phi_H^t$, 
\begin{align*}
\Phi_{g \cdot x}(t)  &= \tau_{g\cdot x}(t) \circ d(g_t \circ \phi_H^t)(g \cdot x(0)) \tau_{g \cdot x}(0)^{-1}\\
&= \tau_{g\cdot x}(t) \circ dg_t(\phi_H^t(g \cdot x(0))) \circ d\phi_H^t(g\cdot x(0)) \circ \tau_{g \cdot x}(0)^{-1}\\
&=\tau_{g\cdot x}(t) \circ dg_t(g_0(x(t))) \circ d\phi_H^t(g\cdot x(0)) \circ \tau_{g \cdot x}(0)^{-1} \\
&=\tau_{g\cdot x}(t) \circ dg_t(g_0(x(t))) \circ \tau_{g_0 \cdot x}(t)^{-1} \circ \tau_{g_0 \cdot x}(t) \circ d\phi_H^t(g\cdot x(0)) \circ \tau_{g \cdot x}(0)^{-1}\\
&= \ell_x(t) \circ \Phi_{g_0 \cdot x}(t).
\end{align*}

We now claim that $\Phi_{g_{0} \cdot x}(t) = \Phi_x(t)$. Observe that since $g_0$ is a Hamiltonian diffeomorphism with the boundary condition \eqref{eq: assumptionong1}, it sends a capping half disk for $x$ to a capping half disk for $g_0\cdot x$. Therefore we may assume that the trivialization $\tau_{g_0 \cdot x}$ along $g_0 \cdot x$ is the same as the pushforward of the trivialization $\tau_x$ along $x$ by $dg_0$. In other words, we may assume that $\tau_{g_0 \cdot x}(t) = \tau_x(t) \circ dg_0(x(t))^{-1}$. Together with the assumption $\phi_H^t \circ g_0 = g_0 \circ \phi_H^t$ again, we compute that
\begin{align*}
\Phi_{g_0 \cdot x}(t) &= \tau_{g_0 \cdot x}(t) \circ d\phi_H^t(g\cdot x(0)) \circ \tau_{g \cdot x}(0)^{-1}\\
&= \tau_x(t) \circ dg_0(x(t))^{-1} d\phi_H^t(g_0 \cdot x(0)) \circ dg_0(x(0)) \circ \tau_x(0)^{-1} \\
&= \tau_x(t) \circ d(g_0^{-1} \circ \phi_H^t \circ g_0)(x(0)) \circ \tau_x(0)^{-1}\\
& = \tau_x(t) \circ d(\phi_H^t)(x(0)) \circ \tau_x(0)^{-1} \\
&= \Phi_x(t).
\end{align*}
Consequently we have
$$
\Phi_{g\cdot x}(t) = \ell_x(t) \circ \Phi_x(t).
$$
Using the loop property of the Robbin--Salamon index, e.g. \cite[Proposition 3.1]{FrKa16}, we conclude 
$$
\mu(g \cdot x) = \mu(x) + I(g)
$$
as asserted.
\end{proof}

The upshot is the following.

\begin{theorem}\label{thm: isomfloerhomSg}
For a path of Hamiltonian diffeomorphisms $g$ on $\widehat W$ with the assumptions \eqref{eq: assumptionong1} and \eqref{eq: assumeg2}, we have an associated graded group isomorphism
$$
S_g : \HF_*(L;H, J) \rightarrow \HF_{* + I(g)}(L;g_*H, g_*J).
$$
\end{theorem}

\subsection{Seidel operator and $\mathcal{L}$-periodic Reeb flows}
Aiming for Theorem \ref{thm: Sgoperatoringeneral}, we now consider a specific path of Hamiltonian diffeomorphisms coming from the Reeb flow on the boundary. Let $(W, \lda)$ be a Liouville domain and $L$ an admissible Lagrangian as before. We assume that the contact boundary $(\Sigma, \xi, \mathcal{L})$ with the Legendrian $\mathcal{L} = \p L$ is product-index-positive, see Definition \ref{def: prodindpos}. In this case, as we have discussed in Section \ref{sec: indfill}, v-shaped wrapped Floer homology can be defined in the symplectization $\R_+ \times \Sigma$ with the Lagrangian $\R_+ \times \mathcal{L}$. We can therefore put $\widehat W = \R_+ \times \Sigma$ and $\widehat L = \R_+ \times \mathcal{L}$ in the construction of the Seidel operator in Section \ref{sec: Seideloperator}. 

Denote the Reeb flow of the contact form $\alpha$ on $\Sigma$ by $\phi_R^{t}$. Then we have a  path of Hamiltonian diffeomorphisms $g_t$ on $\R_+ \times \Sigma$ given by
\begin{equation}\label{eq: defofgperiodic}
g_t (r, y) : = (r, \phi_R^t(y)).
\end{equation}

\begin{remark}
This will be extended to a bit more general case in Section \ref{sec: moregeneralcase}, but to simplify discussions we first work with $g_t$ here.
\end{remark}

\begin{definition}\label{def: lperiodicReebflow}
For a Legendrian $\mathcal{L}$ in a contact manifold $(\Sigma, \alpha)$, the Reeb flow $\phi_R^t$ is called \emph{$\mathcal{L}$-periodic} if there exists $T > 0$ such that $p \in \mathcal{L}$ if and only if $\phi_R^T(p) \in \mathcal{L}$.
\end{definition}
Assume that our Reeb flow is $\mathcal{L}$-periodic. Then the corresponding path $g_t$ in  \eqref{eq: defofgperiodic} satisfies the condition \eqref{eq: assumptionong1}, and it is generated by a (time-independent) Hamiltonian $K^g: \R_+ \times \Sigma \rightarrow \R$ given by
$$
K^g(r, y) = r +c
$$
for some constant $c$.
Let $H: \R_+ \times \Sigma \rightarrow \R$ be a v-shaped admissible Hamiltonian. The pushforward $g_*H: [0,1] \times \R_+ \times \Sigma \rightarrow \R$ is as before given by
$$
g_*H(t, r, y) =H(t, g_t^{-1}(r,y)) + K^g(t,r,y) = h(r) + r + c
$$ 
where $h: \R \rightarrow \R$ is a v-shaped function defining $H$. (Note that $g_*H$ is almost v-shaped but is not constant near $r =0$, see Section \ref{sec: shapeofg_*H}.) By Theorem \ref{thm: isomfloerhomSg}, we have the following isomorphism
$$
S_g : \HF_*(L;H, J) \rightarrow \HF_{*+I(g)}(L;g_*H, g_*J)
$$
for an admissible almost complex structure $J$ on $\R_+ \times \Sigma$.

\subsubsection{The index $I(g)$ and real contact manifolds} \label{sec: compu Ig} Suppose further that the contact type boundary $(\Sigma, \alpha)$ admits a real structure, i.e. an anti-contact involution $\rho: \Sigma \rightarrow \Sigma$, and the Legendrian $\mathcal{L}$ is given by (a connected component of) the fixed point set of $\rho$. Throughout this section, the Reeb flow $\phi_R^t$ is assumed to be periodic. Then it is easy to see that $\phi_R^t$ is $\mathcal{L}$-periodic. In this case, we can compute the index $I(g)$ in terms of the Maslov index of a principal Reeb chord relative to $\mathcal{L}$. 

To define the notion of principal Reeb chords in Definition \ref{def: principalperiod}, we first observe the following. Let $\gamma$ be a periodic Reeb orbit of period $2T$ which starts at $\gamma(0) \in \mathcal{L}$. 

\begin{lemma}\label{lem: halfperiod}
We have $\gamma(T) \in \mathcal{L}$.
\end{lemma}

\begin{proof}
We observe that $\rho(\gamma(T)) = \rho(\phi^{T}_R(\gamma(0))) = \phi^{-T}_R \circ \rho(\gamma(0))$. Since $\gamma(0) \in \mathcal{L}= Fix(\rho)$ we have $\rho(\gamma(T)) = \phi^{-T}_R(\gamma(0))$. Since $\gamma$ has the period equal to $2T$, we see that $\gamma(2T) = \gamma(0)$ or equivalently $\phi^{T}_R(\gamma(0))  = \phi^{-T}_R(\gamma(0))$. It follows that
$$
\rho(\gamma(T)) = \phi^{T}_R(\gamma(0)) = \phi^{T}_R(\gamma(0)) = \gamma(T)
$$
and hence $\gamma(T) \in Fix(\rho) = \mathcal{L}$.
\end{proof}

\begin{definition}\label{def: principalperiod}\
\begin{enumerate}
\item For a periodic Reeb orbit $\gamma$ with period $2T$ and $\gamma(0) \in \mathcal{L}$, we define its \emph{half Reeb chord} $c: [0, T] \rightarrow \Sigma$ by $c(t) : = \gamma(t)$.
\item Let $2 \cdot T_P$ be the least common period of the periodic Reeb flow. A periodic Reeb orbit $\gamma$ is called \emph{principal} if $\gamma$ has the period $2 \cdot T_P$. Its half Reeb chord $c$ (of period $T_P$) is called a \emph{principal} Reeb chord. We call $T_P$ the \emph{principal period} of Reeb chords.
\end{enumerate}
\end{definition}

\begin{remark} \label{rmk: indrel}
In \cite[Proposition 3.1]{KiKwLe18}, the following relation of the indices is given. For a principal Reeb orbit $\gamma$ and its half Reeb chord $c$, we have $2\mu(c) = \mu(\gamma)$.

In addition, since $\mathcal{L}$ itself forms a Morse--Bott connected component consisting of Reeb chords of period $T_P$, the index $\mu(x)$ does not depend on the choice of $c \in \mathcal{L}$ (where the points in $\mathcal{L}$ are identified with the starting point of Reeb chords).
\end{remark}

Now the index $I(g)$ is simply given as follows. 

\begin{lemma}\label{lem: computingI(g)}
The index $I(g)$ is equal to the Maslov index of a principal Reeb chord.
\end{lemma}

\begin{proof}
We recall the definition of the Maslov index $\mu(c)$ of a Reeb chord $c$. Take a trivialization of the contact structure $\xi$ along $c: [0, T] \rightarrow \Sigma$,
$$
\eta_c : c^*\xi \rightarrow [0, T] \times \R^{2(n-1)},
$$
such that $\eta_c (c^*T\mathcal{L}) = \Gamma_0$ for the horizontal Lagrangian $\Gamma_0 \subset \R^{2(n-1)}$. Then the Maslov index $\mu(c)$ is defined by
$$
\mu(c) : = \mu_{\RS}(\Phi_{c} \Gamma_0, \Gamma_0)
$$
where
$$
\Phi_{c} (t) = \eta_c(t) \circ d\phi_R^{t}(c(0))|_{\xi}\circ \eta_c(0)^{-1}.
$$

Note that a Reeb chord $c$ corresponds to a Hamiltonian chord $x := \{r\} \times c$ (for some $r \in \R_+$) of an admissible Hamiltonian on $\R_+ \times \Sigma$. According to the splitting
$$
T(\R \times \Sigma) = \langle {\p_r}, R \rangle \oplus \xi,
$$
where $R$ denotes the Reeb vector field, we can extend the trivialization $\eta_c$ to a trivialization $\tau_x$ of $T(\R_+ \times \Sigma)$ along $x$ by putting, in the obvious way, that
$$
\tau_x|_{\langle {\p_r}, R \rangle} =\id_{\langle {\p_r}, R \rangle} : \langle {\p_r}, R \rangle \rightarrow \R^2.
$$
We also have a splitting of the tangent space of the Lagrangian $L = \R_+ \times \mathcal{L}$ by
$$
TL =T(\R \times \mathcal{L}) = \p_r \oplus T\mathcal{L}. 
$$
The trivialization $\tau_x$ sends $TL$ to the extended Lagrangian $\Lambda_0 : = \R \times  \{0\} \times \Gamma_0 \subset \R^2 \times \R^{2(n-1)} = \R^{2n}$.

The index $I(g)$ of $g$ is by definition given by (note that $g_0 = \id$ in this case)
$$
I(g) = \mu_{\RS}(\ell \Lambda_0, \Lambda_0)
$$
where 
$$
\ell: [0, T] \rightarrow Sp(2n), \quad \ell(t) = \tau_{g \cdot x}(t) \circ dg_t(x(t)) \circ \tau_{x}(t)^{-1} 
$$
and $T$ is the period of $c$.
Since $g_t(r, p) = (r, \phi_R^t)$ by definition, the linearization $d g_t$ acts on $\langle {\p_r}, R \rangle$ trivially. Since $\tau_x$ is extended from $\eta_c$ trivially on $\langle {\p_r}, R \rangle$, we can split the path $\ell$ in such a way that
\begin{equation}\label{eq: splittingofell}
\ell(t) = 
\left[
\begin{array}{c|c}
\tilde\ell_{c} (t) & O \\
\hline
O & \id
\end{array}
\right]
\end{equation}
where $\tilde \ell_x$ is given by
$$
\tilde \ell_{c} (t) = \eta_{\phi_R\cdot c}(t) \circ d\phi_R^{t}(c(t))|_{\xi}\circ \eta_c(t)^{-1}, \quad t \in [0, T]
$$
with $(\phi_R \cdot c) (t) = \phi_R^t(c(t))$.
Applying the same argument as in the proof of Proposition \ref{prop: degreeshift}, we see that 
$$
\mu_{\RS}(\tilde \ell_c\Gamma_0, \Gamma_0) + \mu_{\RS}(\Phi_c \Gamma_0, \Gamma_0) = \mu_{\RS}(\Phi_{\phi_R \cdot c}\Gamma_0, \Gamma_0).
$$
 (In this case, the commuting condition \eqref{eq: assumeg2} obviously holds.) 
 
Now we put $c$ to be a principal Reeb chord. Then the path $\phi_R \cdot c$ is actually a Reeb chord of the (time-scaled) Reeb flow $\phi_R^{2t}$ of the same period as $c$. The Maslov index $\mu_{\RS}(\Phi_{\phi_R \cdot c}\Gamma_0, \Gamma_0)$ of $\phi_R \cdot c$ is therefore the same as the index of the second iteration of $c$, say $c^2$. Since the second iteration $c^2$ is exactly the same as the principal periodic Reeb orbit which has $c$ as the half Reeb chord, we know from \cite[Proposition 3.1]{KiKwLe18} that $\mu(c^2) = 2\mu(c)$, see Remark \ref{rmk: indrel}. It follows that
$$
\mu_{\RS}(\tilde \ell_c\Gamma_0, \Gamma_0) + \mu(c) = 2\mu(c)
$$
and hence $\mu_{\RS}(\tilde \ell_c\Gamma_0, \Gamma_0) = \mu(c)$. By the direct sum property of the Robbin--Salamon index applied to the splitting \eqref{eq: splittingofell}, we conclude that
$$
I(g) = \mu_{\RS}(\ell \Lambda_0, \Lambda_0) = \mu_{\RS}(\tilde \ell_c \Gamma_0, \Gamma_0) = \mu(c).
$$
This completes the proof.
\end{proof}

\subsection{V-shaped wrapped Floer homology using $g_*H$}\label{sec: shapeofg_*H} Let $H: \R_+ \times \Sigma \rightarrow \R$ be a v-shaped admissible Hamiltonian and $J$ an admissible almost complex structure. In this section, we show that the v-shaped wrapped Floer homology can be defined using a cofinal family of Floer data of the form $(g_*H, g_*J)$ even though $g_*H$ is not precisely v-shaped.

\subsubsection{Shape of $g_*H$} 

Recall that the generating Hamiltonian $K^g: \R_+ \times \Sigma \rightarrow \R$ of $g$, defined as \eqref{eq: defofgperiodic}, is given by $
K^g(r, y) = r + c$ for some constant $c$. Taking $c= -1$, the pushforward $g_*H = K^g + H$ has, up to smoothing, the following description, see Figure \ref{fig: vshapedHam} and \ref{fig: pushHam},
\begin{equation}\label{shapegH}
(g_*H)(r)  =  \begin{cases} (\mathfrak{a} + 1) r + \mathfrak{b}_1 -1 & r  \geq 1 \\ 0 & r=1 \\ (-\mathfrak{a} + 1) r + \mathfrak{b}_2 -1 & \delta \leq r \leq 1 \\
r + \mathfrak{c}-1 & r \leq \delta \end{cases} \; \text{where } H(r) = \begin{cases} \mathfrak{a} r + \mathfrak{b}_1 & r  \geq 1 \\ 0 & r= 1 \\ -\mathfrak{a} r + \mathfrak{b}_2 & \delta \leq r \leq 1 \\
\mathfrak{c} & r \leq \delta \end{cases}
\end{equation}
for some constants $\mathfrak{a}, \mathfrak{c} >0, \mathfrak{b}_1, \mathfrak{b}_2$ and sufficiently small $\delta > 0$.

\begin{figure}[htp]
\begin{subfigure}[b]{0.25\textwidth}
\centering
	\begin{overpic}[width=150pt,clip]{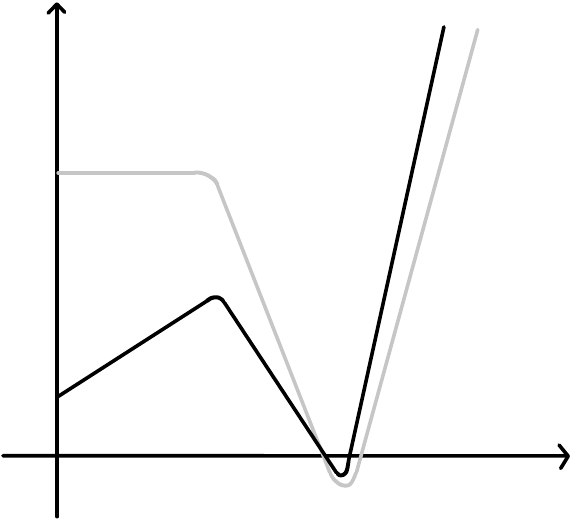} 
	\dashline{3.0}(55,18)(55,58)
	\put(145,5){$r$}
	\put(50,5){$\delta$}
	\put(110,70){$\mathfrak{a}+1$}
	\put(63,48){$\mathfrak{-a}+1$}
	\put(-8,30){$\mathfrak{c}-1$}
	\end{overpic}
	\caption{Shape of $g_*H$}
	\label{fig: pushHam}
	\end{subfigure}\hspace{10.0em}
	\begin{subfigure}[b]{0.25\textwidth}
\centering
	\begin{overpic}[width=150pt,clip]{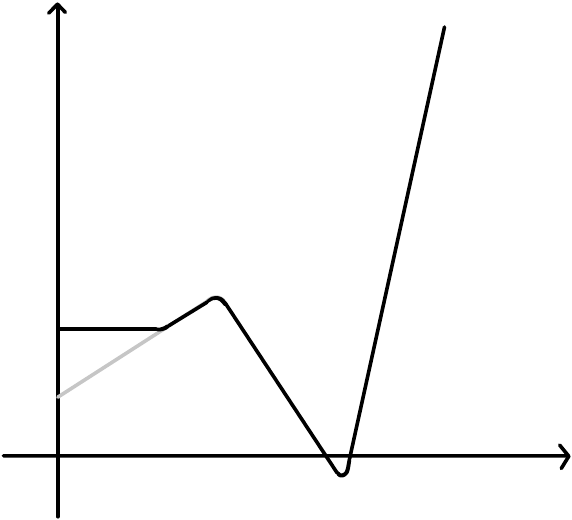} 
	\dashline{3.0}(40,18)(40,50)
	\dashline{3.0}(55,18)(55,58)
	\put(145,5){$r$}
	\put(53,5){$\delta$}
	\put(35,5){$R$}
	\end{overpic}
	\caption{Shape of $(g_*H)_R$}
	\label{fig: trunHam}
	\end{subfigure}
	\caption{}
\end{figure}

For a given action window $(a, b)$, we can take $\delta$ sufficiently small and $\mathfrak{c}$ sufficiently large so that the filtered chain group $\CF^{(a, b)}_*(L;H, J)$ is generated only by the Hamiltonian chords of $H$ in the region $(1-\epsilon, 1+\epsilon) \times \Sigma$ for sufficiently small $\epsilon >0 $. With the same choice of $\delta$ and $\mathfrak{c}$, we have the following lemma. For the sake of simplicity, we abbreviate ``in the region $(1-\epsilon, 1+\epsilon) \times \Sigma$ for sufficiently small $\epsilon >0 $" by ``near $r=1$''. 
\begin{lemma}\label{lem: actionwindowgH}
For the above choice of $\delta$ and $\mathfrak{c}$, the chain group $\CF^{(a+1, b+1)}_*(L;g_*H, g_*J)$ is generated by the Hamiltonian chords of $g_*H$ in the region near $r=1$.
\end{lemma}

\begin{proof}
Since $\widehat L = \R_+ \times \mathcal{L}$ and $\widehat \lda = r \alpha$, the restriction $\lda|_{\widehat L}$ vanishes. Therefore we may assume that the primitive $g_{L}: L \rightarrow \R$ is constant so that the $g_L$-involved terms in the action functional $\mathcal{A}_H$ do not appear. Then we can compute the action of $g \cdot x$ as follows.
\begin{align*}
\mathcal{A}_{g_*H}(g \cdot x)  &= \int_0^1 (g\cdot x)^* \widehat \lda  - \int_0^1 (g_*H)((g\cdot x)(t)) dt \\
&= \int_0^1 x^* \widehat \lda + 1 - \int_0^1 H(x(t)) dt - \int_0^1 K_t^g(x(t)) dt \\
&=\mathcal{A}_H(x) +1 -\int_0^1 K_t^g(x(t)) dt,
\end{align*}
Note that near $r = 1$, we have $K_g= r-1 \approx 0$. We conclude that
$$
\mathcal{A}_{g_*H}(g \cdot x) \approx \mathcal{A}_H(\gamma) +1.
$$
By the choice $\delta$ and $\mathfrak{c}$ with respect to the action window $(a, b)$, the generators $\CF_*(L;g_*H, g_*J)$ (which should come from the region near $r=1$) are in the action window $(a+1, b+1)$.
\end{proof}

\subsubsection{Being constant at the negative end} The shape of $g_*H$ in \eqref{shapegH} is almost v-shaped, but it is not constant at the negative end (i.e. the region in $\R_+\times \Sigma$ with $r \leq \delta$). Analogously to \cite[Section 3.4]{Ue19}, we modify $g_*H$ to make it constant near $r=0$ without changing the (filtered) chain complex $\CF_*(L;g_*H, g_*J)$. Notice that being constant at the negative end is necessary to perform the stretching-the-neck operation, and this is to show that the the v-shaped wrapped Floer homology is defined purely in the symplectization $\R_+ \times \Sigma$.

Let $(g_*H)_R: \R_+ \times \Sigma \rightarrow \R$ be the Hamiltonian which is the same as $g_*H$ in $(R, \infty) \times \Sigma$ and is constant in $(0, R) \times \Sigma$ for some $0<R  <\delta$, see Figure \ref{fig: trunHam}. Let $(a+1, b+1)$ be the action window as in Lemma \ref{lem: actionwindowgH}. Then since $(g_*H)_R$ is the same as $g_*H$ near $r=1$, the Hamiltonian chords of $(g_*H)_R$ in this region are identically the same as those of $g_*H$. Therefore the chain groups $\CF_*^{(a+1, b+1)}(L;g_*H, g_*J)$ and $\CF_*^{(a+1, b+1)}(L;(g_*H)_R, g_*J)$ have the same generators. 

Let $x_-$ and $x_+$ be two generators in $\CF_*^{(a+1, b+1)}(L;g_*H, g_*J) = \CF_*^{(a+1, b+1)}(L;(g_*H)_R, g_*J)$ with index difference by $1$. The proof of the following lemma is completely analogous to the one in \cite[Proposition 3.19]{Ue19}.

\begin{lemma}\label{lem: 1to1Rtruncated} 
Let $x_-,x_+ \in \mathcal{P}_L(H)$ be two Hamiltonian chords such that $|x_+| - |x_-|=1$.
If $R > 0$ is sufficiently small, then there is a bijection
$$
\mathcal{M}(x_-, x_+; g_*H, g_*J) \leftrightarrow \mathcal{M}(x_-, x_+; (g_*H)_R, g_*J).
$$
\end{lemma}

This implies that $\CF^{(a+1, b+1)}_*(L;g_*H, g_*J)$ is the same as $\CF^{(a+1, b+1)}_*(L;(g_*H)_R, g_*J)$ as a chain \emph{complex} for sufficiently small $R > 0$. Therefore their homology groups are isomorphic. Since we have an isomorphism 
$$
S_g : \HF_*^{(a, b)}(L;H, J) \rightarrow \HF_{*+I(g)}^{(a+1, b+1)}(L;g_*H, g_*J),
$$
we obtain an induced isomorphism (with the same notation)
$$
S_g : \HF_*^{(a, b)}(L;H, J) \rightarrow \HF_{*+I(g)}^{(a+1, b+1)}(L;(g_*H)_R, g_*J).
$$

\subsubsection{Taking direct limit} We shall take the direct limit of $\HF_{*}^{(a+1, b+1)}(L;(g_*H)_R, g_*J)$ over a cofinal family of $(g_*H)_R$'s. Note that if two admissible Hamiltonians $H_{-}$ and $H_{+}$ satisfy $H_{+} \leq H_{-}$ , then $(g_*H_+)_R \leq (g_*H_-)_R$ by the construction, so we can describe a family of $(g_*H)_R$'s in terms of $H$'s. We take a family of $H$'s by increasing the slope $\mathfrak{a}$ but keeping $\delta$ sufficiently small so that the Hamiltonian chords within the action window $(a, b)$ are all near $r= 1$. Then the induced family $(g_*H)_R$'s has a property that Hamiltonian chords within the action window $(a+1, b+1)$ are all near $r=1$. In particular, for two Hamiltonians $(g_*H_+)_R \leq (g_*H_-)_R$ in this family, we have a continuation map
$$
f_{H_+,H_{-}}: \HF_{*}^{(a+1, b+1)}(L;(g_*H_+)_R, g_*J_+) \rightarrow \HF_{*}^{(a+1, b+1)}(L;(g_*H_-)_R, g_*J_{-}) .
$$

\begin{lemma}\label{lem: commdiagcontinu}
Let $(H_{-},J_{-})$ and $(H_{+},J_{+})$ be pairs of a v-shaped admissible Hamiltonian and an admissible almost complex structure. Then the following diagram commutes.
\[ \begin{tikzcd}
\HF_*^{(a, b)}(L;H_+, J_+) \arrow{r}{S_g} \arrow[swap]{d}{f_{H_+,H_-}} & \HF_{*+I(g)}^{(a+1, b+1)}(L;(g_*H_+)_R, g_*J_+) \arrow{d}{f_{H_+,H_-}} \\%
\HF_*^{(a, b)}(L;H_-, J_{-}) \arrow{r}{S_g}& \HF_{*+I(g)}^{(a+1, b+1)}(L;(g_*H_-)_R, g_*J_{-})
\end{tikzcd}
\]
\end{lemma}

\begin{proof}
Let $x_{-} \in \PP_{\widehat L}(H_{-})$ and $x_+ \in \PP_{\widehat L}(H_+)$ be Hamiltonian chords with  $|x_{+}| - |x_-|=0$. Choose a regular pair $\{(H_s,J_s)\}_{s\in \R}$ of monotone decreasing family of v-shaped admissible Hamiltonians and admissible almost complex structures such that 
$$
(H_s,J_s) =\begin{cases} (H_{-},J_{-}) & s < -M \\ (H_{+},J_+) &s>M \end{cases}
$$ 
for some large $M>0$. Since the set $\{s\in \R \;|\; \text{$H_s$ is strictly decreasing}\}$  is compact, we can choose a sufficiently small $\delta >0$ such that $H_s$ is constant on $(0,\delta) \times \Sigma$ for all $s \in \R$.

By Lemma \ref{lem: sendfloerstrip} the map $\mathcal{M}(x_{-}, x_+;H_s, J_s) \to \mathcal{M}(g\cdot x_{-}, g\cdot x_+; g_*H_s, g_*J_s)$, $u \mapsto g\cdot u$ is a one-to-one correspondence. In particular $\mathcal{M} (g\cdot x_-, g\cdot x_+; g_*H_s, g_*J_s)$ is compact, and due to the condition $|x_{+} | - |x_-| =0$, it has only finitely many elements, say $u_1, \dots, u_m$. 

Now, to prove the commutativity in the statement, it suffices to show that the moduli spaces $\mathcal{M} (g\cdot x_-, g\cdot x_+; g_*H_s, g_*J_s)$ and $\mathcal{M} (g\cdot x_-, g\cdot x_+; (g_*H_s)_R, g_*J_s)$ are the same. Since all curves in $\mathcal{M}(x_{-}, x_+; H_s, J_s)$ are contained in the symplectization, so are the curves in $\mathcal{M}(g\cdot x_{-}, g\cdot x_+; g_* H_s, g_*J_s)$. Therefore there is a sufficiently small $R >0$ such that the image of $u_j$ lies in $(R,\infty) \times \Sigma$ for all $j =1, \dots, m$. We use this number $R$ for the truncated function $(g_*H)_R$. Then for all $j$, the element $u_j$ actually belongs to $\mathcal{M}(x_-, x_+; (g_*H)_R, g_*J)$ since $(g_*H)_R$ and $g_* H$ coincide in the region $(R,\infty) \times \Sigma$.

For the converse inclusion, we follow the idea in \cite[Proof of Proposition 3.19]{Ue19}. Suppose that the latter moduli space $\mathcal{M}(x_-, x_+; (g_*H)_R, g_*J)$ has an extra element, say $u$. Then the image of $u$ escapes from $(R,\infty) \times \Sigma$. Applying $g^{-1}$ to $u$, we get an element
$$ g^{-1} \cdot u \in \mathcal{M}( x_{-}, x_+; g^*(g_* H_s)_R, g^*(g_*J_s) ),$$
which also escapes from $(R,\infty) \times \Sigma$ since the diffeomorphism $g_t$ preserves the radial coordinate for all $t$. Since $g^*(g_*H)_R = H$ is constant on $(R,\delta) \times \Sigma$, we can apply the stretching-the-neck argument to $g^{-1}u$ at the contact hypersurface $ \{ R \} \times \Sigma$ as in the proof of Proposition \ref{prop: groupindependence}. This leads to a contradiction due to the index-positivity assumption. This completes the proof.
\end{proof}

It is apparent that the direct limit of $\HF_{*}^{(a+1, b+1)}(L;(g_*H)_R, g_*J)$ as $\mathfrak{a} \rightarrow \infty$ is the filtered v-shaped wrapped Floer homology, i.e.
$$
\varinjlim_{\mathfrak{a}}\HF_{*}^{(a+1, b+1)}(L;(g_*H)_R, g_*J) = \widecheck{\HW}_*^{(a+1, b+1)}(L).
$$
This is because, within a fixed action window, the truncated Hamiltonians give the same chain complexes as the genuine v-shaped Hamiltonians. Now by the commutative diagram in Lemma \ref{lem: commdiagcontinu}, the operator $S_g$ induces an isomorphism to the direct limits
$$
S_g : \widecheck{\HW}_*^{(a, b)}(L) \rightarrow \widecheck{\HW}_{*+I(g)}^{(a+1, b+1)}(L).
$$
Taking $\displaystyle\varprojlim_{a}$ followed by $\displaystyle \varinjlim_{b}$, we finally have the following.
\begin{corollary}\label{cor: SgoperatorReebflow}
We have a graded group isomorphism
$$
S_g : \widecheck{\HW}_*(L) \rightarrow \widecheck{\HW}_{*+ I(g)}(L).
$$
\end{corollary}

\subsubsection{More general case}\label{sec: moregeneralcase} In the above, we have defined $\widecheck{\HW}_*(L)$ using Hamiltonians of the form $g_*H$ where $g$ is given by the Reeb flow as \eqref{eq: defofgperiodic}. The same idea basically applies to a bit more general case of $g$ given by
\begin{equation}\label{eq: generalg_tperiodic}
g_t(r, y) = (r, \phi_R^{f(t)}(y))
\end{equation}
for some smooth function $f : [0,1] \to \R$ satisfying $ f(0), f(1) \in T_P \Z$ where $T_P$ is the principal period of Reeb chords (Definition \ref{def: principalperiod}). Note that $g$ in \eqref{eq: generalg_tperiodic} satisfies the conditions \eqref{eq: assumptionong1} and \eqref{eq: assumeg2} for such a function $f$.


A technical difference is that the pushforward $g_*H$ may have a time-\emph{dependent} slope. For this reason, we homotope the given path $g$ to another path $\tilde{g}$ such that the pushforward $\tilde{g}_*H$ still has time-\emph{independent} slope. This in turn provides a cofinal family for the direct limit of $\HF_*(L;\tilde{g}_*H, \tilde{g}_*J)$'s. The following is a straightforward adaptation of \cite[Lemma 3.18]{Ue19}.

\begin{lemma}\label{lem: homotopy} 
Let $g_1$ and $g_2$ be paths of Hamiltonian diffeomorphisms which are homotopic relative to the end points. Let $H_1$ and $H_2$ be v-shaped admissible Hamiltonians such that the slope of $H_2$ is steeper than the slope of $H_1$. Let $J_1$ and $J_2$ be regular almost complex structures. Then there exists a continuation map
$$ \HF(L;g_{1*}H_1,g_{1*}J_1) \to \HF(L;g_{2*}H_2,g_{2*}J_2).$$
\end{lemma}

Suppose we are given a path of Hamiltonian diffeomorphisms 
$ g_t(r, y) = (r, \phi_R^{f(t)}(y)) $
for some smooth function $f : [0,1] \to \R$ satisfying $ f(0),f(1) \in T_P \Z$. Then the Hamiltonian $K^g$ generating $g$ is given by
$$ K^g(t,r,y) = f'(t) r.$$ We observe that the path $g$ is homotopic (relative to the end points) to a path $\tilde{g}$ defined by
$$ \tilde{g}_t (r,y) = (r, \phi_R^{(f(1)-f(0))t+f(0)}(y)).$$
A generating Hamiltonian $K^{\tilde{g}}$ of $\tilde{g}$ is given by
$$ K^{\tilde{g}} (t,r,y) = (f(1) - f(0)) r.$$
Therefore the pushforward $\tilde{g}_* H$ is time-independent for any time-independent Hamiltonian $H$.

Let $\mathfrak{a} < \mathfrak{b}$ be given. By Lemma \ref{lem: homotopy}, we have a map
$$ \HF(L; g_* {H_{\mathfrak{a}}}, g_* J_{\mathfrak{a}}) \to \HF( L;\tilde{g}_*{ H_{\mathfrak{b}} },\tilde{g}_*{ J_{\mathfrak{b}} }),$$
where $H_{\mathfrak{a}}$ is a v-shaped Hamiltonians described in \eqref{shapegH} and $J_{\mathfrak{a}}$, $J_{\mathfrak{b}}$ are regular almost complex structures.
Taking the direct limit as $\mathfrak{a}, \mathfrak{b} \to \infty$, we get a map 
\begin{equation}\label{eq: onedirection} \varinjlim_{\mathfrak{a}}\HF(L;g_*H_{\mathfrak{a}},  g_* J_{\mathfrak{a}} ) \to \varinjlim_{\mathfrak{b}}\HF(L;\tilde{g}_* H_{\mathfrak{b}},\tilde{g}_*{ J_{\mathfrak{b}} }). \end{equation}
Changing the role of $g$ and $\tilde{g}$ in the above, we get the inverse map of (\ref{eq: onedirection}), namely
\begin{equation}\label{eq: theotherdirection} \varinjlim_{\mathfrak{b}}\HF(L;\tilde{g}_* H_{\mathfrak{b}}, \tilde{g}_*{ J_{\mathfrak{b}} }) \to \varinjlim_{\mathfrak{a}}\HF(L;g_*H_{\mathfrak{a}},g_* J_{\mathfrak{a}}), \end{equation} 
by which we conclude the following.




\begin{proposition}
There is an isomorphism
$$ \varinjlim_{\mathfrak{a}}\HF(L;g_*H_{\mathfrak{a}},g_* J_{\mathfrak{a}}) \cong \varinjlim_{\mathfrak{b}}\HF(L;\tilde{g}_* H_{\mathfrak{b}},\tilde{g}_*{ J_{\mathfrak{b}}}).$$
\end{proposition}

Furthermore, since $\tilde{g}_* H_{\mathfrak{b}}$ is time-independent for all $\mathfrak{b}$, the argument in the previous section shows that the latter one $\displaystyle{\varinjlim_{\mathfrak{b}}\HF(L;\tilde{g}_* H_{\mathfrak{b}}, \tilde{g}_* J_{\mathfrak{b}})}$ is isomorphic to the v-shaped wrapped Floer homology $\widecheck{\HW}_*(L)$. We finally have the following.

\begin{theorem}\label{thm: Sgoperatoringeneral}
Let $g$ be a path of Hamiltonian diffeomorphisms on $\R_+ \times \Sigma$ of the form \eqref{eq: generalg_tperiodic}. Then there is an associated graded group isomorphism
$$
S_g : \widecheck{\HW}_*(L) \rightarrow \widecheck{\HW}_{*+ I(g)}(L).
$$
\end{theorem} 

\begin{remark}\label{rmk: notextend}
If the path of Hamiltonian diffeomorphism of the form $\id \times \phi_R^{f(t)}$ on $\R_+ \times \Sigma$ extends to a path of Hamiltonian diffeomorphisms $g'$ on the completion $\widehat W$, then the wrapped Floer homology $\HW_*(L)$ must be finite dimensional. Roughly speaking, this is because taking $g'_*$ to admissible Hamiltonians corresponds to increasing the slope and hence corresponds to taking the direct limit in the definition of the wrapped Floer homology.
In our examples, we already know that $\HW_*(L)$ is infinite dimensional; see Proposition \ref{prop: groupcompu}. As suggested in \cite{Ue19}, we therefore work with v-shaped wrapped Floer homology instead of working directly with the ordinary wrapped Floer homology. 
\end{remark}

\subsection{Module property}\label{sec: productproperty} The following is essentially the same as \cite[Section 3.2]{HuLa10} adapted to open symplectic manifolds as in \cite[Proposition 3.26]{Ue19}.

\begin{lemma}[Homotopy invariance]\label{lem: homotopyproperty}
Let $g$ and $h$ be paths of Hamiltonian diffeomorphisms of the form \eqref{eq: generalg_tperiodic} with the same end points. Suppose that $g$ and $h$ are homotopic relative to the end points. Then $S_g: \widecheck{\HW}_*(L) \rightarrow \widecheck{\HW}_{* + I(g)}(L)$ and $S_h: \widecheck{\HW}_*(L) \rightarrow \widecheck{\HW}_{* + I(h)}(L)$ coincide. 
\end{lemma}


We want to prove the following module property of the operator $S_g$ which is crucial for computing the ring structure. Note, in addition to the starting point, that the end point $g_1$ also commutes with v-shaped admissible Hamiltonians.

\begin{theorem}\label{thm: prodproperty}
The operator $S_g : \widecheck{\HW}_*(L) \rightarrow \widecheck{\HW}_{*+ I(g)}(L)$ satisfies
$$
S_g(x \cdot y) = S_{g_0}(x) \cdot S_g(y) = S_g(x) \cdot S_{g_1}(y)
$$
where $x, y \in \widecheck{\HW}_*(L)$ and $S_{g_j}$ denotes the operator of the constant path $g_j$ for $j=0, 1$.
\end{theorem}

An immediate corollary is the following.
\begin{corollary}\label{cor: constantpathringisom}
If $g$ is a constant path (not necessarily the identity), then $S_g$ is a ring isomorphism on $\widecheck{\HW}_*(L)$ without degree shifting.
\end{corollary}

\begin{proof}[Proof of Theorem \ref{thm: prodproperty}]
We present a proof of the first part, i.e. $S_g(x \cdot y) = S_{g_0}(x) \cdot S_g(y)$. The second part follows from the same argument. We argue analogously to \cite[Proposition 3.8]{HuLa10}, \cite[Proposition 6.3]{Se97}, and \cite[Proposition 3.30]{Ue19}. The idea of the proof is that we choose a specific disk as the domain of half-pair-of-pants taking into account the homotopy property in Lemma \ref{lem: homotopyproperty}. 

Lemma \ref{lem: homotopyproperty} allows us to assume that $g_{[0, 1/4]} \equiv g_0$ since we can homotope $g$ to be constant for $0 \leq t \leq 1/4$ without changing the operator $S_g$. Now we take a specific disk with three points removed on the boundary, namely
$$
\mathcal{S} : = (\R \times [0,1]) \setminus \{(0,0)\}.
$$
For a holomorphic chart near $(0,0)$, we take
$$
(s, t) \mapsto (1/4 e^{-\pi s} \cos(\pi(1-t)), 1/4 e^{-\pi s} \sin(\pi(1-t))
$$
where $(s, t) \in \R_+ \times [0,1]$. Note that the two boundary punctures $(0,0)$ and $s = \infty$ are positive and the other puncture $s = -\infty$ is negative. 

The coefficient $\langle x \cdot y, z \rangle \in \Z_2$, with $|x|+|y| = |z|$, counts the elements of the zero dimensional moduli space $\MM(z, x, y;\beta,H^{\mathcal{S}},J^{\mathcal{S}})$, i.e. half-pair-of-pants from the domain $\mathcal{S}$ with positive ends converging to $x$ and $y$ and the negative end converging to $z$. We arrange that $x$ is the asymptotic at $(0, 0)$, $y$ at $s = \infty$, and $z$ at $s = -\infty$. (Here $(s, t)$ are the coordinates of $\R \times [0,1]$.) 

Let $u$ be an element of $\MM(z, x, y;\beta,H^{\mathcal{S}},J^{\mathcal{S}})$. The mapping $u \mapsto g_* u$ gives one-to-one correspondence
$$
\MM(z, x, y;\beta,H^{\mathcal{S}},J^{\mathcal{S}}) \leftrightarrow \mathcal{M}(g_*z, g_*x, g_*y; \beta, g_*H^{\mathcal{S}}, g_*J^{\mathcal{S}}).
$$
In particular, since $g_{[0,1/4]} \equiv g_0$, we have $g_*x  = g_{0*}x$. So we have a correspondence
$$
\mathcal{M}(z, x, y; \beta, H^{\mathcal{S}}, J^{\mathcal{S}}) \leftrightarrow \mathcal{M}(g_*z, g_{0*}x, g_*y; \beta, g_*H^{\mathcal{S}}, g_*J^{\mathcal{S}}).
$$
This implies that
$$
\langle x \cdot y, z \rangle = \langle g_{0*}x \cdot g_* y, g_*z \rangle,
$$
and hence on the homology level that
$$
\langle x \cdot y, z \rangle = \langle S_{g_0}(x) \cdot S_g (y), S_g (z) \rangle.
$$
Since $S_g$ is an isomorphism of $\Z_2$-modules, we have $\langle S_g (x \cdot y), S_gz \rangle = \langle x \cdot y, z \rangle$. From this we conclude
$$
\langle S_{g_0}(x) \cdot S_g (y), S_g (z) \rangle = \langle S_g (x \cdot y), S_g(z) \rangle,
$$
which implies $ S_{g_0}(x) \cdot S_g (y) = S_g(x \cdot y)$.
\end{proof}

\subsection{Computing the ring structure of v-shaped wrapped Floer homology}\label{subsection: ringstructureofvshaped}
 We now turn to the $A_k$-type Milnor fiber $V
_k$ and the real Lagrangian $L_j$, discussed in Section \ref{sec: realAkfiber}. We can actually work with $W_k$ as in Section \ref{ex: groupcomputation} due to invariance of (v-shaped) wrapped Floer homology under Liouville isotopies. In Section \ref{sec: computegroups}, we have computed the graded group structure of the v-shaped wrapped Floer homology of the real Lagrangian $L_j$, namely Proposition \ref{prop: groupvshapedcompu}.  In Section \ref{sec: indposak} we have shown that the contact boundary of $A_k$-type Milnor fiber is product-index-positive if $n \geq 3$; see Proposition \ref{cor: indexpositicityinexample}. Furthermore the Reeb flow \eqref{eq: periodicReebflow} on the boundary is $\mathcal{L}_j$-periodic for each $j$. We therefore have a well-defined Seidel operator in the v-shaped wrapped Floer homology $\widecheck{\HW}_*(L_j)$ as in Theorem \ref{thm: Sgoperatoringeneral},
$$
S_g : \widecheck{\HW}_*(L_j) \rightarrow \widecheck{\HW}_{*+ I(g)}(L_j)
$$
for each $0\leq j \leq k$. 

Since the index $I(g)$ is given by the Maslov index of principal Reeb chords by Lemma \ref{lem: computingI(g)}, we have
\begin{equation}\label{eq: Igcomp}
\mu(x_P) =  I(g) = 2 + (n-2)(k+1).
\end{equation}
Using these data we can compute the ring structure of $\widecheck{\HW}_*(L_j)$ as follows.

\begin{theorem}\label{thm: computationofringvshaped}
For all $0\leq j\leq k$, $n \geq 3$, and $k \geq 2$, we have a graded ring isomorphism
$$
\widecheck{\HW}_*(L_j) \cong \Z_2[x, y, y^{-1}]/ ( x^2),
$$
where $|x| = I(g) -n +1$ and $|y| = I(g)$.
\end{theorem}

Before proving the above theorem, let us observe the following. A crucial feature of the graded group structure of $\widecheck{\HW}_*(L_j)$ in Proposition \ref{prop: groupvshapedcompu} is that it is of at most rank 1 in each degree. Indeed, the homology groups of the following degrees are of rank one, and the others vanish:
$$
\dots, -I(g) -n +1 ,-I(g), - n +1, 0, I(g) - n +1, I(g), 2I(g) - n +1, 2I(g), 3I(g) - n +1, 3I(g), \dots .
$$
Observe that, starting from the generators of degree $I(g) - n +1$ and $I(g)$, the degrees of the other generators are obtained by shifting them by $\pm I(g)$ repeatedly. We can therefore label the generators in the form either $C_{I(g)m}$ or $C_{I(g)m-n+1}$ for $m \in \Z$, where $C_k$ denotes a generator of the group $\widecheck{\HW}_k(L_j)$ of degree $k$. In particular $C_0$ is the unit element. Since the $S_g$ operator is a group isomorphism with degree shifting by $I(g)$, it follows that
 \begin{equation}\label{eq: generatorofvshaped}
 C_{I(g)m} = S_g^m(C_0)~ \text{ and } ~C_{I(g)m-n+1} = S_{g}^{m}(C_{I(g)-n+1})
\end{equation} 
 for each $m \in \Z$.

\begin{lemma}\label{lem: constidentity}
For the constant path $g_1$, the corresponding operator $S_{g_1}: \widecheck{\HW}_*(L_j) \rightarrow \widecheck{\HW}_*(L_j)$ is the identity.
\end{lemma}

\begin{proof}
This directly follows from the fact that $S_{g_1}$ is a ring isomorphism without grading shift by Corollary \ref{cor: constantpathringisom} and the fact that for all $k \in \Z$, the $k$-th group $\widecheck{\HW}_k(L_j)$ is a one-dimensional vector space over $\Z_2$ if it is not zero. In this situation, we must have $S_{g_1}(C_k) = C_k$ for all $k \in \Z$, i.e. $S_{g_1} = \id$.
\end{proof}


We now prove Theorem \ref{thm: computationofringvshaped}. 

\begin{proof}[Proof of Theorem \ref{thm: computationofringvshaped}]
For degree reasons, if $k \geq 2$, we already know that $C_{I(g)m-n+1} \cdot C_k  = C_k \cdot C_{I(g)m-n+1} = 0$ for all $k \not \in I(g)\Z$ and $m \in \Z$. It only remains to determine the product of the form
$$
C_{I(g)m}\cdot C_k \text{ and } C_k \cdot C_{I(g)m}
$$
for $m \in \Z$ and $k = I(g)\ell$ or $k = I(g)\ell -n +1$ for some $\ell \in \Z$. (Note that the ring $\widecheck{\HW}_*(L_j)$ may not be commutative in general.)

We claim that $C_{I(g)m}\cdot C_k= C_k \cdot C_{I(g)m} = C_{I(g)m + k}$. It suffices to show that $C_{I(g)}\cdot C_k= C_k \cdot C_{I(g)} = C_{I(g) + k}$ since the claim follows by applying this several times. We compute, using Theorem \ref{thm: prodproperty},
\begin{align*}
C_k \cdot C_{I(g)} &= C_k \cdot S_g(C_0) = S_g(C_k \cdot C_0) = S_g(C_k)= C_{I(g)+ k}.
\end{align*}
We next compute, using Theorem \ref{thm: prodproperty} and Lemma \ref{lem: constidentity},
\begin{align*}
C_{I(g)} \cdot C_k &= S_g(C_0) \cdot C_k = S_g(C_0) \cdot S_{g_1}(C_k) = S_g(C_0 \cdot C_k) = S_g(C_k) = C_{I(g) + k}.
\end{align*}
This completes the proof of the claim.

We define a ring homomorphism $\varphi: R \rightarrow \widecheck{\HW}_*(L_j)$ by assigning 
$$
\varphi(x) = C_{I(g)-n+1}, \; \varphi(y) = C_{I(g)}
$$
where $R$ is a ring given by $R = \Z_2[x, y, y^{-1}]$. By the above product computations, we see that $\varphi$ is a ring homomorphism. It also follows from the computation that $\ker \varphi$ is generated by $x^2$. Indeed, we have seen that $C_{I(g) -n +1}^2 = 0$ and $C_{I(g)}C_{-I(g)} = C_0$. We therefore conclude that the ring $\widecheck{\HW}_*(L_j)$ is isomorphic to the quotient ring $\Z_2[x, y, y^{-1}]/(x^2)$ with grading preserved.
\end{proof}

\begin{remark}
The condition $k \geq 2$ is used at the beginning of the proof of Theorem \ref{thm: computationofringvshaped}; If $k=1$ our technique determines the ring structure of $\widecheck{\HW}_*(L_j)$ only partially. 
\end{remark}

\section{Computing the ring structure of wrapped Floer homology}\label{sec: computation}

\subsection{Viterbo transfer map}\label{sec: viterbomap} Let $L$ be an admissible Lagrangian in a Liouville domain $(W,\lambda)$. Following the constructions in \cite{CiFrOa10}, \cite{CiOa18}, we recall a definition of the \emph{Viterbo transfer map} from $\HW_*(L)$ to $\widecheck{\HW}_*(L)$.

Consider an admissible Hamiltonian $K_{\mathfrak{a}}$ on $\widehat W$ which is of the form, up to smoothing,
\begin{equation}\label{shapeofK}
K_{\mathfrak{a}} (r) = \begin{cases} \mathfrak{a}r+\mathfrak{b} & r \geq 1 \\ 0 & r \leq 1 \end{cases}
\end{equation}
for some constants $\mathfrak{a}>0$ and $\mathfrak{b}$. Let $H_{\mathfrak{a}}$ be a v-shaped admissible Hamiltonian of the form $H$ in (\ref{shapegH}) such that $K_{\mathfrak{a}} \leq H_{\mathfrak{a}}$ for $r\leq 1$ and   $K_{\mathfrak{a}} = H_{\mathfrak{a}}$ for $r \geq 1$. Then we can take a non-increasing family of Hamiltonians $\{G_s\}_{s\in \R}$ such that
\begin{equation}\label{familyofhamiltonians}
 G_s = \begin{cases}  H_{\mathfrak{a}} & s \leq -R,\\ K_{\mathfrak{a}} & s \geq R \end{cases}
 \end{equation}
for some sufficiently large $R >0$. We also choose a family of almost complex structures $\{J_s\}_{s \in \R}$ such that
\begin{equation}\label{familyofacs}
 J_s =\begin{cases} J_{H} & s \leq -R \\ J_{K} & s \geq R \end{cases}
 \end{equation}
where $(H_{\mathfrak{a}}, J_{H})$ and $(K_{\mathfrak{a}}, J_K)$ are regular Floer data. These data define a continuation map as in Section \ref{sec: continuationamap}, namely  
$$ f^{(a,b)}_{K_{\mathfrak{a}}, H_{\mathfrak{a}}} : \HF_*^{(a,b)} (L;K_{\mathfrak{a}}, J_{K}) \to \HF_*^{(a,b)} (L;H_{\mathfrak{a}}, J_{H}) $$
for each action window $(a,b)$. 

We define the \emph{Viterbo transfer map with action window $(a,b)$} to be the direct limit of continuation maps 
$$ f^{(a,b)} : \HW^{(a,b)}_* (L) \to \widecheck{\HW}^{(a,b)}_* (L), \quad f^{(a,b)} =\varinjlim_{\mathfrak{a} \to \infty} f^{(a,b)}_{K_\mathfrak{a},H_\mathfrak{a}}.$$
Since the transfer map $f^{(a,b)}$ is essentially a continuation map, which is compatible with enlarging the action window, we can further take the limits as $a \to -\infty$ and $b \to \infty$. We consequently have a map 
$$ 
f  : \HW_* (L) \to \widecheck{\HW}_* (L), \quad f = \lim_{b \to +\infty} \lim_{a\to -\infty} f^{(a,b)},
$$
called the \emph{Viterbo transfer map}.

It is shown in \cite[Theorem 10.2]{CiOa18} that the Viterbo transfer map in symplectic homology respects the ring structure given by pair of pants products. The same argument directly applies to the half-pair-of-pants product, which yields the following.

\begin{proposition}
The Viterbo transfer map $f : \HW_* (L) \to \widecheck{\HW}_* (L)$ is a ring homomorphism.
\end{proposition}

\subsection{Computing the ring structure on $\HW_*(L_j)$}
Let $L_j$ be the real Lagrangian in $A_k$-type Milnor fiber $V_k$, defined in Section \ref{sec: realAkfiber}. We shall show that the Viterbo transfer map $\HW_*(L_j) \rightarrow \widecheck{\HW}_*(L_j)$ is in fact injective. From this, we will get the ring structure of $\HW_*(L_j)$ out of the ring structure of $\widecheck{\HW}_*(L_j)$ which is computed in Theorem \ref{thm: computationofringvshaped}.

More precisely, the positive degree parts of the wrapped Floer homology $\HW_*(L_j)$ and the v-shaped wrapped Floer homology $\widecheck{\HW}_*(L_j)$ can be identified in a sense that both $\HW_*(L_j)$ and $\widecheck{\HW}_*(L_j)$ are isomorphic when $* \geq 0$. We will show that the Viterbo transfer map $f: \HW_*(L_j) \to \widecheck{\HW}_*(L_j)$ maps each generator to a generator of the same degree. This will be done by specifying the Viterbo transfer map by realizing the generators of $\HW_*(L_j)$ and $\widecheck{\HW}_*(L_j)$ in view of Morse--Bott technique.

\subsubsection{Specifying the Viterbo transfer map} As in Section \ref{ex: groupcomputation}, we work with a deformed domain $W_k$ whose boundary is the $A_k$-type Brieskorn manifold $\Sigma_k$. The Reeb flow on the boundary is given as \eqref{eq: periodicReebflow}, and Reeb chords in $(\Sigma_k, \alpha, \mathcal{L}_j)$ are of Morse--Bott type. The v-spectrum is given by \eqref{eq: vspectrum}, and the (ordinary) spectrum of Reeb chords is obtained by taking positive periods in the v-spectrum, i.e.
$$
\Spec(\Sigma_k, \alpha, \mathcal{L}_j) = \{N \cdot 2(k+1)\pi \;|\; N \in \N\}.
$$
Recall that for each $T \in \Spec(\Sigma_k, \alpha, \mathcal{L}_j) \subset \widecheck{\Spec}(\Sigma_k, \alpha, \mathcal{L}_j)$, the corresponding Morse--Bott submanifold $\mathcal{L}_T$ of Reeb chords forms $S^{n-1}$-family. 

Let $0 < \mathfrak{a} \not \in \Spec(\Sigma_k, \alpha, \mathcal{L})$. We take a v-shaped admissible Hamiltonian $K_{\mathfrak{a}}$ and an admissible Hamiltonian $H_{\mathfrak{a}}$ on $\widehat W_k$ defined as in Section \ref{sec: viterbomap}. Their Hamiltonian 1-chords form Morse--Bott submanifolds in $L$ since the Reeb chords on the boundary are of Morse--Bott type. Note that $K_{\mathfrak{a}}$ and $H_{\mathfrak{a}}$ are chosen to be identical to each other in the region where the Morse--Bott families of Hamiltonian 1-chords corresponding to non-constant Reeb chords appear (roughly, the region ``$r \geq 1$''). Therefore the Morse--Bott families of Hamiltonian 1-chords of $K_{\mathfrak{a}}$ in this region can be canonically identified with the Morse--Bott families of Hamiltonian 1-chords of $H_{\mathfrak{a}}$. Furthermore they both correspond to Morse--Bott families $\mathcal{L}_T$ of Reeb chords for $T \in \Spec(\Sigma_k, \alpha, \mathcal{L}_j)$ with $0 < T < \mathfrak{a}$.

By taking Morse functions on each Morse--Bott submanifold of Hamiltonian 1-chords we can perturb $K_{\mathfrak{a}}$ and $H_{\mathfrak{a}}$ in a standard way (e.g. as in \cite{Po94}) to make them non-degenerate. Denote the resulting perturbed Hamiltonians by $\widetilde K_{\mathfrak{a}}$ and $\widetilde H_{\mathfrak{a}}$ respectively. The Hamiltonian 1-chords of those correspond to critical points of Morse functions on Morse--Bott submanifolds. In particular, in the region where $K_{\mathfrak{a}} = H_{\mathfrak{a}}$, we may take the same perturbing Morse functions for both. Then $\widetilde K_{\mathfrak{a}}$ and $\widetilde H_{\mathfrak{a}}$ still coincide in this region and hence the Hamiltonian 1-chords in there canonically correspond to each other.

Since $\mathcal{L}_T$ forms an $S^{n-1}$-family, we can take  perturbing Morse functions with exactly two critical points; the minimum and the maximum. As a result, for each $T \in \Spec(\Sigma_k, \alpha, \mathcal{L})$ with $T < \mathfrak{a}$, we have two non-degenerate Hamiltonian 1-chords, say $x_T$ (minimum) and $y_T$ (maximum), of $\widetilde K_{\mathfrak{a}}$ and two Hamiltonian 1-chords, say $x'_T$ (minimum) and $y'_T$ (maximum), of $\widetilde H_{\mathfrak{a}}$. Even though we denote them with different notations, they are actually the same chords since $\widetilde K_{\mathfrak{a}} = \widetilde H_{\mathfrak{a}}$ where the chords appear. Moreover the gradings of $x'_T$ and $y'_T$ are the same as those of $x_T$ and $y_T$, respectively, for each $T$. 

\begin{proposition}\label{prop: computationofcontinuation}
 The continuation map $f_{ \widetilde{K}_{\mathfrak{a}}, \widetilde{H}_{\mathfrak{a}} }$ at the chain level  satisfies
$$ f_{\widetilde{K}_{\mathfrak{a}},\widetilde{H}_{\mathfrak{a}}} (x_T) = x'_T~ \text{ and  }~ f_{\widetilde{K}_{\mathfrak{a}},\widetilde{H}_{\mathfrak{a}}}(y_T) = y'_T$$
for $T \in \text{Spec} (\Sigma_k, \alpha,\mathcal{L}_j)$ with $T < \mathfrak{a}$.
\end{proposition}
\begin{proof}
We prove the assertion for $x_T$ and $x'_T$, and the argument for $y_T$ and $y_T'$ is verbatim. 

Let $u \in \mathcal{M} ( x_T, x'_T ; G_s,J_s)$ be an element of the parametrized moduli space in the definition of $f_{ \widetilde{K}_{\mathfrak{a}}, \widetilde{H}_{\mathfrak{a}} } $. We claim that the curve $u : \R \times [0,1] \to V_k$ stays at $x_T =x'_T$, i.e. $u(s,t) = x_T(t) = x'_T(t)$ for all $s\in \R$. This implies that there is a unique element of $\mathcal{M} ( x_T, x'_T ; G_s,J_s)$ and hence the assertion will follow.

To prove the claim, we estimate the energy of $u$ as follows.
\begin{align*}
E(u) &= \frac{1}{2} \int_{ \R \times [0,1] } | du - X_{G_s} \otimes dt|^2 ds \wedge dt
= \int_{\R \times [0,1]}  u^* \omega - (u^* dG_s) \wedge dt\\
&= \int_{\R \times [0,1]} d (u^* \lambda - u^* G_s dt) + \frac{\p G_s}{\p s} ds \wedge dt
\leq  \int_{\R \times [0,1]} d (u^* \lambda - u^*G_s dt)&\\
&= \int_{\p (\R \times [0,1])} u^* \lambda - u^*G_s dt
= \mathcal{A}_{\widetilde{K}_{\mathfrak{a}}} ( x_T) - \mathcal{A}_{\widetilde{H}_{\mathfrak{a}}} ( x'_T) =0.
\end{align*}
Here, the inequality follows from the fact that $\frac{\p G_s}{\p s} \leq 0$. The last equality follows from the fact that both Hamiltonians $\widetilde{K}_{\mathfrak{a}}$ and  $\widetilde{H}_{\mathfrak{a}}$ are equal in the region where the chords appear. The computation shows that the energy $E(u)$ must be zero, which implies that $du - X_{G_s} dt$ is constantly zero and hence that $u(s,t) = x_T(t) = x'_T(t)$ for all $s \in \R$. This completes the proof.
\end{proof}

In view of the computation of the $E^1$-page of the spectral sequence in Section \ref{ex: groupcomputation}, namely differentials vanish for degree reasons, the 1-chords $x'_T$ and $y'_T$ (with $T < \mathfrak{a}$) are precisely the generators of the group $\HF_k(L_j; \widetilde{H}_{\mathfrak{a}})$ of degree $k > 0$, and likewise the 1-chords $x_T$ and $y_T$ are the generators of the group $\HF_k(L_j; \widetilde{K}_{\mathfrak{a}})$ of degree $k >0$. Therefore Proposition \ref{prop: computationofcontinuation} says that the Viterbo transfer map
$$
f_{ \widetilde{K}_{\mathfrak{a}}, \widetilde{H}_{\mathfrak{a}} } : \HF_*(L_j; \widetilde{K}_{\mathfrak{a}}) \rightarrow \HF_*(L_j; \widetilde{H}_{\mathfrak{a}})
$$
at homology level is the obvious inclusion.



To understand what happens when we take the direct limit $\lim \HF_{*}(L_j;\widetilde{K}_{\mathfrak{a}})$ as $\mathfrak{a} \to \infty$, we choose a specific cofinal family of Hamiltonians. First take an increasing sequence $\{ \mathfrak{a}_n\;|\; n\in \N\}$ of slopes at cylindrical ends such that
$\displaystyle{\lim_{n\to \infty} \mathfrak{a}_n = \infty}$ and the distance between $\mathfrak{a}_n$ and the spectrum $\text{Spec}(\Sigma_k, \alpha, \mathcal{L}_j)$ is greater than a number, say $\epsilon_0$, for every $n\in \N$.
We now inductively define a sequence of Hamiltonians $\{K_{\mathfrak{a_n}}\}_{n \in \N}$. First we choose an admissible Hamiltonian $K_{\mathfrak{a}_1}$ such that $K_{\mathfrak{a}_1}(r,y) = k_1 (r)$ for some function $k_1$ in the cylindrical coordinate $r$ satisfying 
\begin{itemize}
\item $k_1''(r) \geq 0$
\item $k_1 '(r) =\begin{cases} 0 & r \leq 1 \\ a_1 & r \geq 1+\frac{\epsilon_0}{2}. \end{cases}$
 \end{itemize}
Suppose we have chosen the Hamiltonians $K_{\mathfrak{a}_k}$ for $k <n$. We choose the $n$-th Hamiltonian $K_{\mathfrak{a}_n}$ in such a way that $K_{\mathfrak{a}_n}(r,y) = k_n (r)$ for some function $k_n$ in the cylindrical coordinate $r$ satisfying 
 \begin{itemize}
\item $k_n''(r) \geq 0$
 \item $ k_n(r) =k_{n-1}(r) $ for $r \leq 1+ \epsilon_0 \cdot( \sum_{i=1}^{n-1} \frac{1}{2^i})$.
 \item $k_n '(r) = a_n$ for $r \geq 1+\epsilon_0 \cdot( \sum_{i=1}^n \frac{1}{2^i}) ).$
 \end{itemize}
 
Then the sequence $\{ K_{\mathfrak{a}_n} \}_{n \in \N}$ form a cofinal family of admissible Hamiltonians. A key point of the construction is that $K_{\mathfrak{a}_m} = K_{\mathfrak{a}_n}$ for all $m \leq n$ in the region $\{ 1\leq r \leq 1+\epsilon_0 \cdot (\sum_{i=1}^m \frac{1}{2^i}) \}$. So we can perturb each Hamiltonian the sequence $\{ K_{\mathfrak{a}_n} \}_{n \in \N}$ to get a cofinal family of non-degenerate Hamiltonians $\{\widetilde{K}_{\mathfrak{a}_n} \}_{n \in \N}$ in such a way that $\widetilde{K}_{\mathfrak{a}_m} = \widetilde{K}_{\mathfrak{a}_n}$ for all $m \leq n$ in the region $\{ 1\leq r \leq 1+\epsilon_0 \cdot (\sum_{i=1}^m \frac{1}{2^i}) \}$. In this case, in view of the proof of Proposition \ref{prop: computationofcontinuation}, the continuation map $f_{m,n} : \HF_{*}(L_j;\widetilde{K}_{\mathfrak{a}_m}) \to \HF_{*}(L_j;\widetilde{K}_{\mathfrak{a}_n})$ is the obvious inclusion. It follows that the homology classes of $x_T$ and $y_T$ in $\HF_*(L_j;\widetilde{K}_{\mathfrak{a}_n})$ survive in the process of taking the direct limit as $n \to \infty$ and hence define nonzero homology classes in $\HW_*(L_j)$. 

Analogous arguments for v-shaped Hamiltonians show that Hamiltonian chords $x'_T$ and $y'_T$ define nonzero homology classes in $\widecheck{\HW}_*(L_j)$. As a result we have:

\begin{lemma}\label{lemma: generators}
Let $N$ be a positive integer. Denote $T_0 : = 2 (k+1) \pi$.
\begin{enumerate}
\item The homology class of $x_{N T_0}$ and $y_{ NT_0}$ generate $\HW_{\{(n-2)(k+1) + 2\} N-n+1}(L_j)$ and\\ $\HW_{\{(n-2)(k+1) + 2\} N } (L_j)$ respectively.
\item  The homology class of $x'_{NT_0}$ and $y'_{ NT_0}$ generate $\widecheck{\HW}_{\{(n-2)(k+1) + 2\} N-n+1}(L_j)$ and\\ $\widecheck{\HW}_{\{(n-2)(k+1) + 2\} N} (L_j)$ respectively.\end{enumerate}
\end{lemma}

Denote by $B_d$ the homology class of a generator of $\HW_d (L_j)$ for $d \geq 0$. Lemma \ref{lemma: generators} implies that $x_{NT_0}$ is a representative of $B_{\{(n-2)(k+1) + 2\} N -n+1}$ and $y_{NT_0}$ is a representative of $B_{\{(n-2)(k+1) + 2\} N}$ for $N >0$. The group $\HW_0(L_j)$ of degree zero is generated by the unit element $B_0$. Likewise $x'_{NT_0}$ is a representative of $C_{\{(n-2)(k+1) + 2\} N-n+1}$ and $y'_{NT_0}$ is a representative of $B_{\{(n-2)(k+1) + 2\} N}$ for $N \in \Z$, where $C_j$ is defined in (\ref{eq: generatorofvshaped}).

Combining Lemma \ref{lemma: generators} with Proposition \ref{prop: computationofcontinuation}, we have the following.

\begin{theorem}\label{thm: computationoftransfermap}
The Viterbo transfer map $f : \HW_*(L_j) \to \widecheck{\HW}_*(L_j)$ is given by
$$ B_d \mapsto C_d$$
for each $d =  \{(n-2)(k+1) + 2\} N$ with $N \geq 0$ and $d=\{(n-2)(k+1) + 2\} N -n+1$ with $N \geq 1$. In particular, the Viterbo transfer map is an injective ring homomorphism.
\end{theorem}
\begin{proof}
For the last statement, recall that the wrapped Floer homology $\HW(L_j)$ is generated by $B_0$ and $B_d$'s for $d= \{(n-2)(k+1) + 2\} N$ and $d=\{(n-2)(k+1) + 2\} N -n+1$, $N>0$ as a $\Z_2$-module, see Proposition \ref{prop: groupcompu}.
\end{proof}

\subsubsection{The ring structure of the wrapped Floer homology $\HW_*(L_j)$}
We now compute the ring structure of the wrapped Floer homology $\HW_*(L_j)$. The following computation is analogous to Theorem \ref{thm: computationofringvshaped}. (The index $I(g)$ below is given in \eqref{eq: Igcomp}.)

\begin{theorem}\label{thm: computationofring}

For $0 \leq j \leq k$, $n \geq 3$, and $k \geq 2$, we have a graded ring isomorphism
$$\HW_*(L_j) \cong \Z_2[x,y]/(x^2) $$
where $|x| = I(g) - n+1$ and $|y| = I(g)$.
\end{theorem}
\begin{proof}

From Theorem \ref{thm: computationofringvshaped} and Theorem \ref{thm: computationoftransfermap}, we observe
$$ B_{M I(g) } \cdot B_{N I(g)} = B_{(M+N) I(g) }$$
for all $M,N \geq 0$. Indeed, we can compute
\begin{align*}
f(B_{M I(g) } \cdot B_{N I(g)}) = f(B_{M I(g) }) \cdot f( B_{N I(g)}) = C_{M I(g) } \cdot C_{N I(g)} = C_{(M+N) I(g)} = f(B_{(M+N) I(g)}).
\end{align*}

The first equality follows from the fact that $f$ is a ring homomorphism. The second equality follows from Theorem \ref{thm: computationoftransfermap}, and the third equality follows from Theorem  \ref{thm: computationofringvshaped}. 
Finally the assertion follows from the fact that the transfer map $f$ is injective.

Similarly we see that 
$$ B_{M I(g)} \cdot B_{N I(g) -n+1} = B_{ (M+N)I(g) -n+1} =  B_{N I(g) -n+1} \cdot B_{M I(g)}$$
for all $M\geq 0$ and $N \geq 1$. Then, the ring homomorphism $\Z_2[x,y]/(x^2) \to \HW(L_j)$ defined by
$$ x\mapsto B_{I(g) -n+1}, y\mapsto B_{I(g)}$$
gives an isomorphism.
\end{proof}

\begin{remark}
If $k=1$, then $V_k$ can be identified with $T^*{S^n}$, and $L_j$ with $j =0,1$ is a cotangent fiber in $T^*S^n$. Therefore by the result in \cite{AbPoSc08}, we know that $\HW_*(L_j) \cong \Z_2[x]$ as graded rings with $|x| = n-1$.
\end{remark}

\bibliographystyle{abbrv}
\bibliography{biblography}

\end{document}